\documentclass[11pt, leqno]{amsart}
\usepackage{geometry}
\usepackage{graphicx}
\usepackage{amsmath, amssymb, mathtools}
\usepackage[square, authoryear]{natbib}
\usepackage{enumitem}
\usepackage{subcaption}
\usepackage{epstopdf}
\usepackage{empheq}
\usepackage{xcolor}
\usepackage{bbm}
\usepackage{hyperref}
\usepackage{makecell}

\input{auxiliary}
\theoremstyle{plain}
\newtheorem{theorem}{Theorem}[section]
\newtheorem{lemma}[theorem]{Lemma}

\newtheorem{corollary}[theorem]{Corollary}

\theoremstyle{definition}
\newtheorem{definition}[theorem]{Definition}
\newtheorem*{notation}{Notation}

\theoremstyle{remark}

\newtheorem{remark}[theorem]{Remark}

\numberwithin{equation}{section}

\makeatletter
\newcommand{\subalign}[1]{%
  \vcenter{%
    \Let@ \restore@math@cr \default@tag
    \baselineskip\fontdimen10 \scriptfont\tw@
    \advance\baselineskip\fontdimen12 \scriptfont\tw@
    \lineskip\thr@@\fontdimen8 \scriptfont\thr@@
    \lineskiplimit\lineskip
    \ialign{\hfil$\m@th\scriptstyle##$&$\m@th\scriptstyle{}##$\crcr
      #1\crcr
    }%
  }
}
\makeatother

\DeclareMathOperator{\arcosh}{arcosh}

\title[Fast SGL Fourier transforms for scattered data]{Fast SGL Fourier transforms for scattered data\addtitlefootnote}

\definetitlefootnote{The corresponding C++ implementation is available from \url{https://github.com/cwuelker/SGLPack}.}

\author{Christian W\"ulker}
\address{Christian W\"ulker - Department of Mechanical Engineering, Johns Hopkins University}
\email{christian.wuelker@jhu.edu}

\date{\today}

\begin{document}

\begin{abstract}
Spherical Gauss-Laguerre (SGL) basis functions, \textit{i.\,e.}, normalized functions of the type $L_{n-l-1}^{(l + 1/2)}(r^2) r^l Y_{lm}(\vartheta,\varphi)$, $|m| \leq l < n \in \mathbb{N}$, $L_{n-l-1}^{(l + 1/2)}$ being a generalized Laguerre polynomial, $Y_{lm}$ a spherical harmonic, constitute an orthonormal polynomial basis of the space $L^2$ on $\mathbb{R}^3$ with radial Gaussian (multivariate Hermite) weight $\exp(-r^2)$. We have recently described fast Fourier transforms for the SGL basis functions based on an exact quadrature formula with certain grid points in $\mathbb{R}^3$. In this paper, we present fast SGL Fourier transforms for scattered data. The idea is to employ well-known basal fast algorithms to determine a three-dimensional trigonometric polynomial that coincides with the bandlimited function of interest where the latter is to be evaluated. This trigonometric polynomial can then be evaluated efficiently using the well-known non-equispaced FFT (NFFT). We prove an error estimate for our algorithms and validate their practical suitability in extensive numerical experiments.
\end{abstract}


\maketitle

\section{Introduction}
\label{sec:introduction}
Let $\|\cdot\|_2$ denote the standard Euclidean norm on $\mathbb{R}^3$. We consider the weighted $L^2$ space
\begin{equation*}
H \,\coloneqq\, \left\lbrace f : \mathbb{R}^3 \to \mathbb{C} ~\textnormal{Lebesgue measurable and} \int_{\mathbb{R}^3} |f(x)|^2 \, \mathrm{e}^{-\|x\|_2^2} \, \mathrm{d}x \,<\, \infty \right\rbrace
\end{equation*}
equipped with the inner product
\begin{equation*}
\langle f, g \rangle_H \,\coloneqq \int_{\mathbb{R}^3} f(x)\, \overline{g(x)}\, \mathrm{e}^{-\|x\|_2^2}\, \mathrm{d}x, \quad\quad f, g \in H,
\end{equation*}
and induced norm $\|\cdot\|_H \coloneqq \sqrt{\langle\cdot,\cdot\rangle_H}$. As the (classical) multivariate Hermite polynomials, spherical Gauss-Laguerre basis functions are orthogonal polynomials in the Hilbert space $H$. They arise from a particular construction approach in spherical coordinates. The latter are defined as \emph{radius} $r \in [0,\infty)$, \emph{polar angle} $\vartheta \in [0, \pi]$, and \emph{azimuthal angle} $\varphi \in [0, 2\pi)$, connected with Cartesian coordinates $x$, $y$, and $z$ via
\begin{align*}
x \,&=\, r \sin\vartheta \cos\varphi,\\
y \,&=\, r \sin\vartheta \sin\varphi,\\
z \,&=\, r \cos\vartheta.
\end{align*}

\begin{definition}
The \emph{spherical Gauss-Laguerre} (SGL) \emph{basis function} of orders $n \in \mathbb{N}$, $l \in \{0,\dots,n-1\}$, and $m \in \{-l,\dots,l\}$ is defined as
\begin{equation*}
H_{nlm} \,:\, \mathbb{R}^3 \,\to\, \mathbb{C}, \quad\quad H_{nlm}(r,\vartheta,\varphi) \,\coloneqq\, N_{nl} \, R_{nl}(r) \, Y_{lm}(\vartheta,\varphi),
\end{equation*}
wherein
\begin{equation*}
N_{nl} \,\coloneqq\, \sqrt{\frac{2(n-l-1)!}{\Gamma(n+1/2)}}
\end{equation*}
is a normalization constant, $Y_{lm}$ is the spherical harmonic of degree $l$ and order $m$ (see \citep[Sect.\,1.6.2]{dai_xu}, for example), while the radial part is defined as
\begin{equation*}
R_{nl}(r) \,\coloneqq\, L_{n-l-1}^{(l+1/2)}(r^2)\, r^l,
\end{equation*}
$L_{n-l-1}^{(l+1/2)}$ being a generalized Laguerre polynomial (see, \textit{e.\,g.}, \citep[Sect.\,6.2]{andrews_askey_roy}).
\end{definition}

\begin{theorem}[{\citep[Cor.\,1.3]{prestin_wuelker}}]\label{thm:completeness}
The SGL basis functions constitute a complete orthonormal polynomial set (a polynomial orthonormal basis) in the Hilbert space $H$.
\end{theorem}

Theorem \ref{thm:completeness} implies that a function $f \in H$ can be approximated arbitrarily well w.\,r.\,t.\ $\|\cdot\|_H$ by finite linear combinations of the SGL basis functions. Such linear combinations are referred to as \emph{bandlimited} functions. In particular, a function $f \in H$ is called bandlimited with bandwidth $B$ if the \emph{SGL Fourier coefficients} $\hat{f}_{nlm} \coloneqq \langle f, H_{nlm} \rangle_H$ vanish for $n > B$. We have recently described fast and reliable SGL Fourier transforms, \textit{i.\,e.}, generalized FFTs for the SGL basis functions \citep{prestin_wuelker}. These fast algorithms compute the $B(B+1)(2B+1)/6$ potentially non-zero SGL Fourier coefficients $\hat{f}_{nlm}$ of a function $f$ with bandwidth $B$ in $\mathcal{O}(B^4)$ or even only $\mathcal{O}(B^3 \log^2 B)$ computation steps from $(2B)^3$ sampled function values, instead of the naive $\mathcal{O}(B^6)$ computation steps (as usual, we define a single \emph{computation step} as a complex multiplication and subsequent addition). In addition to our fast SGL Fourier transforms, another advantage of using the SGL basis functions is that their spectral behavior under rotations and translations in $\mathbb{R}^3$ is completely known \citep{prestin_wuelker_2}. This allows to efficiently solve certain three-dimensional matching problems (cf.\ \citep[Problem 2.3]{prestin_wuelker_2}). However, in our previously described fast SGL Fourier transforms, any respective function of interest is to be sampled at certain grid points in $\mathbb{R}^3$.

There are many conceivable applications in which the sample values $f(x_i)$ of a bandlimited function $f \in H$ are not given on the grid points in \citep[Thm.\,3.4]{prestin_wuelker}. It could be, for example, that the points $x_i$ constitute a Cartesian grid in $\mathbb{R}^3$, or that these points are \emph{scattered} in another way. In such cases, computation of the SGL Fourier coefficients $\hat{f}_{nlm}$ of $f$ using our previously described fast SGL Fourier transforms after interpolation of the given function values $f(x_i)$ is generally not advisable. Therefore, in this paper, we develop \emph{non-lattice fast SGL Fourier transforms} (NFSGLFTs).

In the case of the classical FFT, the now standard non-equispaced equivalent (NFFT) was introduced by \citet{potts_steidl_tasche}, subsequent to works including \citep{dutt_rokhlin, dutt_rokhlin_2, beylkin, elbel_steidl} (cf.\ \citep[Appx.\,D]{keiner_kunis_potts}). Nowadays the NFFT is firmly established in practice and continues to prove itself useful in many applications.
A direct employment of the NFFT also yields ``non-equispaced" FFTs on the two-dimensional unit sphere $\mathbb{S}^2$ \citep{kunis_potts} and on the three-dimensional rotation group $\textnormal{SO}(3)$ \citep{potts_prestin_vollrath}. In both cases, the idea is to use well-known basal fast algorithms to determine a multivariate trigonometric polynomial that coincides with the bandlimited function of interest where the latter is to be evaluated. This trigonometric polynomial can then be evaluated efficiently using the NFFT. In this paper, we pursue the same strategy in the SGL case, resulting in a class of NFSGLFTs analogous to the generalized NFFTs on $\mathbb{S}^2$ and $\textnormal{SO}(3)$ mentioned above. In addition to the three-dimensional NFFT of \citeauthor{potts_steidl_tasche}, as basal fast algorithms, we use a fast discrete Legendre transform (FLT), the Clenshaw-Smith algorithm or, alternatively, a fast discrete polynomial transform (FDPT), as well as the well-known fast discrete cosine transform (DCT).

Analogously as in the derivation of all the above-mentioned (generalized) NFFTs, we begin the derivation of our NFSGLFTs with the discrete transform that reconstructs function values $f(x_i)$ of a bandlimited function $f \in H$ with bandwidth $B$ at $M$ scattered points $x_i \in \mathbb{R}^3$ from given SGL Fourier coefficients $\hat{f}_{nlm}$, $|m| \leq l < n \leq B$. To state this transform, we linearize the index range of the SGL basis functions: We identify the triple $[n, l, m]$, $|m| \leq l < n \leq B$, with $\mu \in \lbrace 0, \dots, B(B+1)(2B+1)/6 - 1\rbrace$ via the one-to-one correspondence
\begin{equation*}
\mu \,=\, \frac{n(n-1)(2n-1)}{6} + l(l+1) + m.
\end{equation*}
This allows for understanding the indices $n$, $l$, and $m$ as functions of the linear index $\mu$,
\begin{empheq}[left=\empheqlbrace]{alignat=2}\label{eq:n_mu_l_mu_m_mu}
\,\,&n(\mu) \,&&=\, \left\lfloor \frac{1}{\sqrt{3}} \cosh\left(\frac{1}{3} \arcosh \left(36 \sqrt{3} \mu\right) \right) + \frac{1}{2} \right\rfloor, \vphantom{\left\lfloor \sqrt{\frac{()}{6}} \right\rfloor} \notag\\
\,\,&l(\mu) \,&&=\, \bigg\lfloor \sqrt{\mu - \frac{n(\mu) (n(\mu)-1) (2n(\mu)-1)}{6}} \bigg\rfloor, \\
\,\,&m(\mu) \,&&=\, \mu - \frac{n(\mu) (n(\mu)-1) (2n(\mu)-1)}{6} - l(\mu)(l(\mu) + 1). \vphantom{\left\lfloor \sqrt{\frac{()}{6}} \right\rfloor}\notag
\end{empheq}

\begin{definition}[NDSGLFT]\label{def:ndsglft}
Let scattered points $x_{i} \in \mathbb{R}^3$, $i = 0,\dots,M-1$, and a bandwidth $B$ be given. We set
\begin{equation*}
\Lambda 
\,=\, \Lambda(B; x_0, \dots, x_{M-1})
\,\coloneqq\,
\begin{bmatrix}
H_{n(\mu),l(\mu),m(\mu)}(x_i)
\end{bmatrix}_{\subalign{i &= 0, \dots, M - 1 \\ \mu &= 0, \dots, B(B+1)(2B+1)/6-1}} \in\, \mathbb{C}^{M \times B(B+1)(2B+1)/6}.
\end{equation*}
The corresponding linear mapping $\Lambda : \mathbb{C}^{B(B+1)(2B+1)/6} \to \mathbb{C}^M$ is called the \emph{non-lattice discrete SGL Fourier transform} (NDSGLFT).
\end{definition}

Let now a Fourier vector $\hat{f} \coloneqq [\hat{f}_{n(\mu),l(\mu),m(\mu)}]_{\mu = 0, \dots, B(B+1)(2B+1)/6-1}$ of a bandlimited function $f \in H$ with bandwidth $B$, as well as scattered points $x_0, \dots, x_{M - 1} \in \mathbb{R}^3$ be given. Without risk of confusion, we define the vector $f \coloneqq [f(x_i)]_{i = 0, \dots, M - 1}$ containing the corresponding function values of the function $f$. These function values $f$ can be computed from the Fourier vector $\hat{f}$ by multiplication of the latter by the transformation matrix $\Lambda$, \textit{i.\,e.},
\begin{equation}\label{eq:ndsglft}
f \,=\, \Lambda \hat{f}.
\end{equation}
A direct multiplication by the matrix $\Lambda$, however, apparently requires $\mathcal{O}(MB^{3})$ computation steps. In practice, this is prohibitively expensive. The NFSGFTs presented in this paper, in contrast, have an asymptotic complexity of $\mathcal{O}(\Phi(B) + \Psi(B)M)$, where $\Phi$ and $\Psi$ are functions of the bandwidth $B$,
\begin{empheq}[]{alignat=2}
\,\,&\Phi(B) \,&&=\, B^4 + (\sigma(B) B)^3 \log (\sigma(B) B),\notag\\[-8pt]
& && \label{eq:Phi_Psi}\\[-8pt]
\,\,&\Psi(B) \,&&=\, q(B)^3,\notag
\end{empheq}
where $\sigma = \sigma(B)$ is the oversampling factor and $q = q(B)$ the cutoff parameter of the employed three-dimensional NFFT (Sect.\,\ref{sec:nfft}). When using an FDPT instead of the Clenshaw-Smith algorithm, $\Phi(B) = B^3 \log^2 B + (\sigma(B) B)^3 \log (\sigma(B) B)$ can be achieved. Note that the above functions $\Phi$ and $\Psi$ are representatives of larger function classes.

The NFSGLFTs presented in this paper are characterized by a respective factorization of the matrix $\Lambda$ in \eqref{eq:ndsglft}. This means that we also have a respective fast \emph{adjoint} transform, \textit{i.\,e.}, a fast algorithm for multiplication with the Hermitean-transposed matrix $\Lambda^\textnormal{H}$, with the same complexity. Analogously as in the case of the NFFT, we can thus realize a fast \emph{inverse} transform (iNFSGLFT), \textit{i.\,e.}, a fast algorithm for computing the SGL Fourier coefficients $\hat{f}_{nlm}$, $|m| \leq l < n \leq B$, of a bandlimited function $f \in H$ with bandwidth $B$ from $M$ given scattered data $f(x_0), \dots, f(x_{M-1})$, as an iterative conjugate-gradient (CG) method.

It is important to note that as is the case in the classical NFFT and its above-mentioned generalizations, the NFSGLFTs presented in this paper are \emph{approximating} algorithms, the results of which are approximative even in exact arithmetics. Hence, in our investigations, the relation between the approximation error and the functions $\Phi$ and $\Psi$ plays a crucial role. In particular, we will show that is possible to appropriately choose $\sigma = \textnormal{const.}$ and $q = o(B)$ in order to control the error when the bandwidth $B$ is increasing.

The remainder of this paper is organized as follows: Sections \ref{sec:nfft}, \ref{sec:dct}, \ref{sec:flt}, and \ref{sec:clenshaw} deal with the required NFFT, DCT, FLT, and the Clenshaw-Smith algorithm\,/\,FDPT, respectively. The derivation of the NFSGLFTs is contained in Section \ref{sec:nfsglft}. In Section \ref{sec:error}, we prove an error estimate. Finally, in Section \ref{sec:numerical_experiments}, we report and discuss our extensive numerical results, demonstrating the practical applicability of our fast algorithms. Table \ref{tab:complexity} provides an overview of the arithmetic and storage complexity of the fast algorithms considered in this work.

\begin{table}\small
\caption{Arithmetic and storage complexity of the fast algorithms considered in this paper. We here consider the well-known NFFT of \citep{potts_steidl_tasche}, and state the storage complexity for the FLT variant of \citep{healy_et_al} and the FDPT variant of \citep{driscoll_healy_rockmore}, respectively. For simplicity, in the Clenshaw-Smith algorithm and the FDPT we assume that the number of target points equals the bandwidth.}\label{tab:complexity}
\begin{tabular}{ l | l | c | c }
 & \textbf{Parameters} & \makecell{\textbf{Arithmetic}\\\textbf{complexity}} & \makecell{\textbf{Storage}\\\textbf{complexity}} \\\hline
\makecell[l]{Clenshaw-Smith \\ algorithm} & \makecell[l]{$n$: bandwidth, \\ \hphantom{$n$: }number of target points} & $\mathcal{O}(n^2)$ & $\mathcal{O}(n)$ \\\hline
DCT & \makecell[l]{$n$: bandwidth} & $\mathcal{O}(n \log n)$ & $\mathcal{O}(n)$ \\\hline
FDPT & \makecell[l]{$n$: bandwidth, \\ \hphantom{$n$: }number of target points} & $\mathcal{O}(n \log^2 n)$ & $\mathcal{O}(n \log n)$ \\\hline
FFT & \makecell[l]{$d$: dimension \\ $n$: bandwidth} & $\mathcal{O}(n^d \log n)$ & $\mathcal{O}(n^d)$ \\\hline
FLT & \makecell[l]{$n$: bandwidth} & \makecell[c]{$\mathcal{O}(n^2 \log^2 n)$\\(semi-naive variant: $\mathcal{O}(n^3)$)} & \makecell[c]{$\mathcal{O}(n^2 \log^2 n)$\\($\mathcal{O}(n^3)$)} \\\hline
FSGLFT & \makecell[l]{$B$: bandwidth} & $\mathcal{O}(B^3 \log^2 B)$ & $\mathcal{O}(B^3)$ \\\hline
NFFT & \makecell[l]{$d$: dimension \\ $n$: bandwidth \\ $m$: number of target points \\ $\sigma$: oversampling factor \\ $q$: cutoff parameter}  & $\mathcal{O}((\sigma n)^d \log(\sigma n) + q^d m)$ & $\mathcal{O}((\sigma n)^d)$ \\\hline
NFSGLFT & \makecell[l]{$B$: bandwidth \\ $M$: number of target points \\ $\sigma$: oversampling factor \\ $q$: cutoff parameter} & $\mathcal{O}(B^3 \log^2 B + (\sigma B)^3 \log (\sigma B) + q^3 M)$ & $\mathcal{O}((\sigma B)^3)$ \\
\end{tabular}
\end{table}

\begin{notation}
We write $f \lesssim g$ if there exists a constant $C > 0$ such that $f(x) \leq C g(x)$ for all $x$.
\end{notation} 

\section{Non-equispaced fast Fourier transform (NFFT)}
\label{sec:nfft}
In this section, we review the functional principle of the NFFT of \citeauthor{potts_steidl_tasche}, due to its fundamental importance within this work. We follow the outline of \citep[Sect.\,1.1]{potts}. As an important result, in Theorem \ref{thm:nfft_error}, we further derive an error estimate for the $d$-dimensional NFFT that we will later need.
Although the NFFT itself is a grid-free transform, certain grids occur in the functional background; they also describe the index range of the Fourier coefficients. We begin with the definition of these grids, as well as the $d$-dimensional torus.

\begin{definition}\label{def:nfft_grid}
Let $d, n \in \mathbb{N}$, $n$ even. We set
\begin{equation*}
I^d_n \,\coloneqq\, \lbrace z = [\zeta_0, \dots, \zeta_{d-1}] \in \mathbb{Z}^d \,:\, - n / 2 \leq \zeta_i < n / 2, ~ i = 0, \dots, d - 1  \rbrace.
\end{equation*}
\end{definition}

\begin{definition}
For $d \in \mathbb{N}$, the \emph{$d$-dimensional torus} is defined as the quotient group
\begin{equation*}
\mathbb{T}^d \,\coloneqq\, \mathbb{R}^d / 2 \pi \mathbb{Z}^d.
\end{equation*}
\end{definition}

\begin{remark}
For each equivalence class $[t] \in \mathbb{T}^d$, the representative $t$ can be chosen in $[0,2\pi)^d$. We thus simply identify the torus $\mathbb{T}^d$ with the $d$-dimensional hypercube $[0,2\pi)^d$.
\end{remark}

The derivation of the NFFT starts with the reconstruction of function values from given Fourier coefficients. Let such Fourier coefficients $\omega_k \in \mathbb{C}$, $k \in I^d_n$, of a $d$-dimensional trigonometric polynomial
\begin{equation}\label{eq:nfft_p}
p(t) \,\coloneqq\, \sum\nolimits_{k \in I^d_n} \omega_k \, \mathrm{e}^{\mathrm{i} \langle k, t \rangle_2}, \quad\quad t \in \mathbb{T}^d,
\end{equation}
of degree at most $n$ be given ($n$ even). The NFFT is a fast algorithm to evaluate the trigonometric polynomial $p$ at $m$ scattered points $t_0, \dots, t_{m-1} \in \mathbb{T}^d$. The difference with the classical FFT is that these points do not necessarily lie on a grid, and that their number $m$ is independent of $n$.

To bring the above problem into matrix-vector notation, the index range of the Fourier coefficients of $p$ is linearized (cf.\ \citep[p.\,11]{potts}): We identify $k = [\kappa_{0},\dots,\kappa_{d-1}] \in I^d_n$ with $\chi \in \lbrace 0, \dots, n^d - 1\rbrace$ via the bijective relation
\begin{equation*}
\chi \,=\, \left(\kappa_0 + \frac{n}{2}\right) + n \left(\kappa_1 + \frac{n}{2}\right) + \cdots + n^{d-1} \left(\kappa_{d-1} + \frac{n}{2}\right).
\end{equation*}
This allows for understanding $k$ as a function of $\chi$,
\begin{equation*}
k(\chi) \,=\, 
\begin{bmatrix}
\left\lfloor\frac{\chi - (\chi \textnormal{ mod } n^{j+1})}{n^j} \right\rfloor - \frac{n}{2}
\end{bmatrix}_{j = 0, \dots, d - 1}.
\end{equation*}
The computation of the polynomial values $p(t_0), \dots, p(t_{m-1})$ is equivalent to evaluating the matrix-vector product
\begin{equation}\label{eq:nfft_matrix}
\begin{bmatrix}
p(t_i) 
\end{bmatrix}_{i = 0, \dots, m - 1}
\,=\,
\underbrace{
\begin{bmatrix}
\mathrm{e}^{\mathrm{i}\langle k(\chi), \, t_i \rangle_2}
\end{bmatrix}_{\subalign{i &= 0, \dots, m - 1\\ \chi &= 0, \dots, n^d - 1}}
}_{\eqqcolon\, N}
\cdot
\begin{bmatrix}
\omega_{k(\chi)} 
\end{bmatrix}_{\chi = 0, \dots, n^d - 1}.
\end{equation}

\begin{definition}[NDFT]
The above linear mapping $N = N(d; n; t_0, \dots, t_{m-1}) : \mathbb{C}^{n^d} \to \mathbb{C}^m$ is called the \emph{$d$-dimensional non-equispaced discrete Fourier transform} (NDFT).
\end{definition}

A direct evaluation of the product \eqref{eq:nfft_matrix} requires $\mathcal{O}(mn^d)$ steps, in accordance with the size of the matrix $N$. The NFFT is an efficient approximating algorithm with a lower complexity. The idea is the following: Determine an approximant $s$ to the trigonometric polynomial $p$ that can be evaluated efficiently in the local (time) domain, so that $p$ and $s$ are as similar as possible in the frequency domain. The latter requirement is met as best as possible if
\begin{equation}\label{eq:NFFT_Fourier_comparison}
(2\pi)^{-d} \int_{\mathbb{T}^d} s(t) \, \mathrm{e}^{- \mathrm{i} \langle k, t \rangle_2} \, \mathrm{d}t \,=\, \omega_k, \quad\quad k \in \mathbb{Z}^d,
\end{equation}
holds, where $\omega_k \coloneqq 0$ for $k \in \mathbb{Z}^d \setminus I^d_n$.

Let now $\varphi \in L^{1}(\mathbb{R}^{d}) \cap L^{2}(\mathbb{R}^{d})$ be an ansatz function that can be evaluated in $\mathcal{O}(1)$ steps. We assume that the $2\pi$-periodized version
\begin{equation*}
\tilde{\varphi}(t) \,\coloneqq\, \sum\nolimits_{z \in \mathbb{Z}^d} \varphi(t - 2 \pi z), \quad\quad t \in \mathbb{T}^d,
\end{equation*}
has a uniformly converging Fourier series. The latter is given by
\begin{equation}\label{eq:tilde_varphi_Fourier}
\tilde{\varphi}(t) \,=\, \sum\nolimits_{k \in \mathbb{Z}^d} \tilde{\varphi}_{k} \, \mathrm{e}^{\mathrm{i} \langle k, t \rangle_2}, \quad\quad t \in \mathbb{T}^d,
\end{equation}
with the Fourier coefficients
\begin{equation*}
\tilde{\varphi}_{k} \,=\, (2\pi)^{- d} \int_{\mathbb{T}^d} \tilde{\varphi}(t) \, \mathrm{e}^{- \mathrm{i} \langle k, t \rangle_2} \, \mathrm{d}t, \quad\quad k \in \mathbb{Z}^d.
\end{equation*}
The equality in \eqref{eq:tilde_varphi_Fourier} is to be understood in the $L^{2}$ sense; we shall not mention this for the following Fourier series. The Fourier coefficients $\tilde{\varphi}_{k}$ are directly connected with the continuous Fourier transform of the ansatz function $\varphi$ via the well-known Poisson summation formula, \textit{i.\,e.},
\begin{equation}\label{eq:nfft_phi_tilde_Fourier}
\tilde{\varphi}_k \,=\, (2\pi)^{-d} \int_{\mathbb{R}^d} \varphi(x) \, \mathrm{e}^{- \mathrm{i} \langle k, x \rangle_2} \, \mathrm{d}x, \quad\quad k \in \mathbb{Z}^d.
\end{equation}

As a first approximant $\tilde{p}$ to the trigonometric polynomial $p$, a linear combination of translates of the $2\pi$-periodized ansatz function $\tilde{\varphi}$ is chosen,
\begin{equation}\label{eq:p_1}
\tilde{p}(t) \,\coloneqq\, \sum\nolimits_{l \in I^d_{\sigma n}} \alpha_l \, \tilde{\varphi} \left(t - \frac{2 \pi}{\sigma n} \, l \right), \quad\quad t \in \mathbb{T}^d,  
\end{equation}
where $\sigma \in \mathbb{N}$ is an \emph{oversampling factor}, and the scaling coefficients $\alpha_{l}$ of the translates are to be determined so that \eqref{eq:NFFT_Fourier_comparison} approximately holds.
Expanding $\tilde{p}$ into a Fourier series yields
\begin{align}
\tilde{p}(t) 
\,&=\, \sum\nolimits_{k \in \mathbb{Z}^d} \, \beta_k \, \tilde{\varphi}_k \, \mathrm{e}^{\mathrm{i} \langle k, t \rangle_2}\notag\\\label{eq:p_1_Fourier}
\,&=\, \sum\nolimits_{k \in I^d_{\sigma n}} \beta_k \, \tilde{\varphi}_k \, \mathrm{e}^{\mathrm{i} \langle k, t \rangle_2} \,+\,\sum\nolimits_{z \in \mathbb{Z}^d \setminus \lbrace 0 \rbrace} \sum\nolimits_{k \in I^d_{\sigma n}} \beta_k \, \tilde{\varphi}_{k + \sigma n z} \, \mathrm{e}^{\mathrm{i} \langle k + \sigma n z, t \rangle_2} 
\end{align}
with the $\sigma n$-periodic coefficients
\begin{equation}\label{eq:nfft_beta_k}
\beta_{k} \,=\, \sum\nolimits_{l \in I^d_{\sigma n}} \alpha_l \, \mathrm{e}^{- 2 \pi \mathrm{i} \langle k, l \rangle_2 / \sigma n}.
\end{equation}

Assuming that the absolute value of the Fourier coefficients $\tilde{\varphi}_{k + \sigma n z}$ with $k \in I^d_{\sigma n}$ and $z \in \mathbb{Z}^d \setminus \lbrace 0 \rbrace$ is negligibly small, the double sum on the right-hand side of \eqref{eq:p_1_Fourier} can be neglected. 
Under the additional assumption that the absolute value of the Fourier coefficients $\tilde{\varphi}_k$ does not vanish for $k \in I^d_n$, a comparison of the first sum on the right-hand side of \eqref{eq:p_1_Fourier} with the right-hand side of \eqref{eq:nfft_p} motivates the particular choice
\begin{equation*}
\beta_k \,\coloneqq\,
\begin{cases}
\omega_k / \tilde{\varphi}_k & \textnormal{if } k \in I^d_n,\\
0 & \textnormal{if } k \in I^d_{\sigma n} \setminus I^d_n,
\end{cases}
\end{equation*}
so that the equality in \eqref{eq:NFFT_Fourier_comparison} holds for all $k \in I^d_{\sigma n}$.

Now the scaling coefficients $\alpha_l$ in \eqref{eq:p_1} are determined from the coefficients $\beta_k$. To this end, \eqref{eq:nfft_beta_k} is brought into matrix-vector notation:
\begin{equation*}
\begin{bmatrix}
\beta_{k(\chi)} 
\end{bmatrix}_{\chi = 0, \dots, (\sigma n)^d - 1}
\,=\,
\underbrace{
\begin{bmatrix}
\mathrm{e}^{- 2 \pi \mathrm{i} \langle k(\chi),\, l(\nu) \rangle_2 / \sigma n}
\end{bmatrix}_{\chi, \nu = 0, \dots, (\sigma n)^d - 1}
}_{\eqqcolon\, F}
\cdot
\begin{bmatrix}
\alpha_{l(\nu)} 
\end{bmatrix}_{\nu = 0, \dots, (\sigma n)^d - 1}.
\end{equation*}
The matrix $F = F(d,n,\sigma)$ is a classical Fourier matrix in $d$ dimensions with special ordering of the indices. Its inverse is given by $(\sigma n)^{-d} F^{\textnormal{H}}$. In particular, this implies that by means of the $d$-dimensional iFFT, the scaling coefficients $\alpha_l$ can be computed in $\mathcal{O}((\sigma n)^{d} \log (\sigma n))$ steps from the coefficients $\beta_{k}$. To obtain a very fast algorithm, the oversampling factor $\sigma$ should thus be chosen such that $\sigma n$ is a power of two.

In a last step, the approximant $\tilde{p}$ is approximated by the final approximant $s$ that can be evaluated in local space more efficiently. Under the assumption that the ansatz function $\varphi$ decays fast in local space, it can be replaced by another function $\psi$, the support of which is contained in a hypercube $[-2\pi q/\sigma n, 2\pi q/\sigma n]^d$ ($q \in \mathbb{N}$, $q < \sigma n$). In particular, the choice
\begin{equation}\label{eq:nfft_psi}
\psi(x) \,\coloneqq\,
\begin{cases}
\varphi(x) & \textnormal{for } x \in [-2\pi q/\sigma n, 2\pi q/\sigma n]^d,\\
0 & \textnormal{otherwise},
\end{cases}
\end{equation}
is made, and again the $2\pi$-periodized version $\tilde{\psi} \coloneqq \sum\nolimits_{z \in \mathbb{Z}^d} \psi(\,\cdot-2 \pi z)$ is considered. The number $q$ is called the \emph{cutoff parameter}. As an approximant to $\tilde{p}$, and thus to $p$, the function
\begin{equation*}
s(t) \,\coloneqq\, \sum\nolimits_{l \in I^d_{\sigma n}} \alpha_l \, \tilde{\psi} \left(t - \frac{2 \pi}{\sigma n} \, l\right), \quad\quad t \in \mathbb{T}^d,  
\end{equation*}
lends itself particularly well (cf.\ Eq.\,\ref{eq:p_1}). In accordance with \eqref{eq:nfft_psi}, for fixed $t = [\tau_0, \dots, \tau_{d-1}]$, the index range $I^d_{\sigma n}$ can be restricted to the range $\lbrace l = [\iota_0, \dots, \iota_{d-1}] \in I^d_{\sigma n} : \sigma n \tau_i / 2\pi - q \leq  \iota_i \leq \sigma n \tau_i / 2\pi + q, ~ i = 0, \dots, d - 1 \rbrace$, as can be checked geometrically easily. This allows for an evaluation of the approximant $s$ in $\mathcal{O}(q^d)$ instead of $\mathcal{O}((\sigma n)^d)$ steps. The overall arithmetic complexity of the above-described NFFT is, therefore, $\mathcal{O}((\sigma n)^d \log (\sigma n) + mq^d)$. The storage complexity amounts to $\mathcal{O}((\sigma n)^d)$.

The above-derived NFFT corresponds to an approximate factorization of the matrix $N$ in \eqref{eq:nfft_matrix}. Without going into detail, we here refer to \citep[Sect.\,1.2]{potts}. This means in particular that we also have an \emph{adjoint} NFFT, \textit{i.\,e.}, a fast algorithm for multiplication with the Hermitean-transposed (adjoint) matrix $N^\textnormal{H}$. The adjoint NFFT has the same arithmetic and storage complexity as the NFFT.

In this work, we will employ a Gaussian ansatz function $\varphi$ (cf.\ \citep[Sect.\,1.3.2]{potts}):

\begin{definition}\label{def:nfft_gauss}
Let $\sigma, n \in \mathbb{N}$, $n$ even, and $\lambda > 0$. We define the \emph{Gaussian ansatz function} for the NFFT as
\begin{equation*}
\varphi(x) \,\coloneqq\, \prod_{j = 0}^{d - 1} \phi(\xi_j), \quad\quad x = [\xi_0, \dots, \xi_{d - 1}] \in \mathbb{R}^d,
\end{equation*}
with the underlying univariate function
\begin{equation*}
\phi(\xi) \,\coloneqq\,
\frac{1}{\sqrt{2 \pi \lambda}} \left(\frac{\sigma n}{2 \pi}\right) \mathrm{e}^{- (\sigma n / 2 \pi)^2 \, \xi^2 / 2 \lambda}.
\end{equation*}
\end{definition}

We can state a closed-form expression for the continuous Fourier transform of the Gaussian ansatz function at points in $\mathbb{Z}^d$ by generalizing \citep[Eq.\,1.34]{potts} to the $d$-dimensional case. By \eqref{eq:nfft_phi_tilde_Fourier}, this facilitates the use of the Gaussian ansatz function in the NFFT.

\begin{lemma}\label{lemma:nfft_Gauss_Fourier}
Let\/ $\sigma$, $n$, and\/ $\lambda$ be as in Definition \ref{def:nfft_gauss}. For the continuous Fourier transform of the Gaussian ansatz function, we have that
\begin{equation*}
\int_{\mathbb{R}^d} \varphi(x) \, \mathrm{e}^{- \mathrm{i} \langle k, x \rangle_2} \, \mathrm{d}x \,=\, \prod_{j = 0}^{d - 1} \phi_{\kappa_j}, \quad\quad k =  [\kappa_0, \dots, \kappa_{d - 1}] \in \mathbb{Z}^d,
\end{equation*}
with\/ $\phi_\kappa \coloneqq \exp(- 2 \lambda (\pi \kappa / \sigma n)^2)$.
\end{lemma}

In the next part of this section, we derive an error estimate for the NFFT with Gaussian ansatz function. Specifically, we give an upper bound for the \emph{maximum absolute error}
\begin{equation*}
E_\infty \,=\, E_\infty(d, n; \sigma, q; t_0, \dots, t_{m - 1}; w) \,\coloneqq\, \max_{i \in \lbrace 0, \dots, m - 1 \rbrace} |p(t_i) - s(t_i)|
\end{equation*}
with given Fourier coefficients $w \coloneqq [\omega_{k(\chi)}]_{\chi = 0, \dots, n^d - 1}$. For this, by making use of Lemma \ref{lemma:nfft_Gauss_Fourier}, we generalize \citep[Thm.\,1.7]{potts} to the $d$-dimensional case (note the different normalization of the Gaussian ansatz function in \citep{potts}). We need this error bound in Section \ref{sec:error} for $d = 3$.

\begin{theorem}\label{thm:nfft_error}
Choosing $\sigma \geq (\sqrt{d} + 1)/2$, the maximum absolute error of the NFFT with Gaussian ansatz function and the special choice\/ $\lambda \coloneqq \sigma q / (2\sigma - 1) \pi$ is bounded by
\begin{equation*}
E_\infty
\,\leq\, \|w\|_1 \, \bigg((2^d - 1) \left(2 + \frac{1}{\pi q}\right)^d + \frac{3^d - 1}{q^{d / 2}} \bigg(\sqrt{\frac{2 \sigma - 1}{2 \sigma}} + \frac{1}{2\pi} \sqrt{\frac{2 \sigma}{2 \sigma - 1}} \bigg)^d \bigg) \, \mathrm{e}^{- q \pi \left(1 - \frac{1}{2\sigma}\left(1 + \frac{d}{2 \sigma - 1}\right)\right)}.
\end{equation*}
In the special case\/ $d = 3$, we obtain for\/ $\sigma \geq 2$ due to monotonicity the error estimate
\begin{equation*}
E_\infty \,\lesssim\, \|w\|_1 \, \mathrm{e}^{- q \pi / 2}.
\end{equation*}
\end{theorem}

\begin{proof}
We first get with the triangle inequality for arbitrary $i \in \lbrace 0,\dots,m-1 \rbrace$ that
\begin{equation}\label{eq:nfft_error_1}
|p(t_i) - s(t_i)| \,\leq\, \underbrace{|p(t_i) - \tilde{p}(t_i)|}_{\eqqcolon\, E_{\textnormal{a}}(t_{i})} \,+\, \underbrace{|\tilde{p}(t_i) - s(t_i)|}_{\eqqcolon\, E_{\textnormal{t}}(t_i)}.
\end{equation}
The \emph{aliasing error} $E_{\textnormal{a}}(t_{i})$, which results from the cutoff of $\tilde{p}$ in the frequency domain, can be estimated as \citep[Eq.\,1.23]{potts}
\begin{equation*}
E_\textnormal{a}(t_i) \,\leq\, \|w\|_{1} \, \max_{k \in I^d_n} \, \sum\nolimits_{z \in \mathbb{Z}^d \setminus \lbrace 0 \rbrace} \left|\frac{\tilde{\varphi}_{k + \sigma n z}}{\tilde{\varphi}_{k}}\right|.
\end{equation*}
On the right-hand side, we distinguish between summands with exactly one of the $d$ components of $z$ different from zero, with exactly two of the $d$ components of $z$ different from zero, etc. It follows with \eqref{eq:nfft_phi_tilde_Fourier} and Lemma \ref{lemma:nfft_Gauss_Fourier} that
\begin{align*}
E_\textnormal{a}(t_i) \,&\leq\, \|w\|_{1} \, \sum_{j = 1}^d {d \choose j} \bigg(\underbrace{\max_{\kappa \in I^1_n} \, \sum\nolimits_{\zeta \in \mathbb{Z} \setminus \lbrace 0 \rbrace} \left|\frac{\phi_{\kappa + \sigma n \zeta}}{\phi_{\kappa}}\right|}_{\eqqcolon\, L} \bigg)^j.
\end{align*}
The proof of \citep[Thm.\,1.7]{potts} reveals that
\begin{align*}
L \,=\, L(\sigma, n, q) \,&\leq\, \mathrm{e}^{-2 \pi^2 \lambda (1 - 1/\sigma)} \left(1 + \frac{\sigma}{2 \pi^2 \lambda (2 \sigma - 1)} + \mathrm{e}^{- 4 \lambda \pi^2 / \sigma} \left(1 + \frac{\sigma}{2 \pi^2 \lambda (2 \sigma + 1)} \right) \right)\\
\,&\leq\, \mathrm{e}^{- q \pi \left(1 - \frac{1}{2 \sigma}\left(1 + \frac{d}{2 \sigma - 1}\right)\right)} \left(1 + \frac{1}{2 \pi q} + \mathrm{e}^{- 4 \pi q / (2 \sigma - 1)} \left(1 + \frac{1}{2 \pi q} \frac{2 \sigma - 1}{2 \sigma + 1} \right) \right)\\
\,&\leq\, \mathrm{e}^{- q \pi \left(1 - \frac{1}{2 \sigma} \left(1 + \frac{d}{2 \sigma - 1} \right) \right)} \left(2 + \frac{1}{\pi q}\right),
\end{align*}
due to the special choice of $\lambda$. It follows with the binomial formula that
\begin{equation}\label{eq:nfft_error_a}
E_\textnormal{a}(t_i) \,\leq\, \|w\|_{1} \, (2^d - 1) \left(2 + \frac{1}{\pi q}\right)^d \mathrm{e}^{- q \pi \left(1 - \frac{1}{2 \sigma}\left(1 + \frac{d}{2 \sigma - 1} \right) \right)}.
\end{equation}
Here, we have used that $\sigma \geq (\sqrt{d} + 1)/2$ and therefore
\begin{equation*}
\mathrm{e}^{- q \pi \left(1 - \frac{1}{2 \sigma}\left(1 + \frac{d}{2 \sigma - 1} \right) \right)} \,\leq\,1.
\end{equation*}

The \emph{truncation error} $E_{\textnormal{t}}(t_i)$ that is due to the cutoff of $\varphi$ in local space can be estimated as
\begin{equation}\label{eq:nfft_E_t}
E_\textnormal{t}(t_i) \,\leq\, \|w\|_{1} \, (\sigma n)^{-d} \, \max_{k \in I^d_n} \, |\tilde{\varphi}_k|^{-1} \, \sum\nolimits_{l \in \mathbb{Z}^d \,:\, \left\|t_i + \frac{2 \pi}{\sigma n} l\right\|_\infty \,\geq\, \frac{2 \pi}{\sigma n} q} \left|\varphi \left(t_i + \frac{2 \pi}{\sigma n}l \right) \right|
\end{equation}
\citep[Eq.\,1.27]{potts}. According to \eqref{eq:nfft_phi_tilde_Fourier} and Lemma \ref{lemma:nfft_Gauss_Fourier}, we have that
\begin{equation*}
\max_{k \in I^d_n} \, |\tilde{\varphi}_k|^{-1} 
\,=\, (2 \pi)^d \left(\max_{\kappa \in I^1_n} \, |\phi_\kappa|^{-1} \right)^d 
\,=\, (2 \pi)^d \, \mathrm{e}^{\lambda (\pi / \sigma)^2 d / 2}.
\end{equation*}
Without loss of generality, let $t_i = [\tau_0, \dots, \tau_{d - 1}] \in [0, 2 \pi / \sigma n)^d$. On the right-hand side of \eqref{eq:nfft_E_t}, we distinguish between summands with multi-index $l = [\iota_j]_{j = 0, \dots, d - 1}$ for which the condition $|\tau_j + 2 \pi \iota_j / \sigma n| \geq 2 \pi q / \sigma n$ is fulfilled by exactly one $j$ of $d$, by exactly two $j$ of $d$, etc. It is
\begin{equation*}
\left|\phi \left(\tau_j + \frac{2 \pi}{\sigma n}\iota\right) \right|
\,\leq\,
\begin{cases}
\left|\phi\left(\frac{2 \pi}{\sigma n}\iota\right)\right| & \textnormal{if } \iota \geq 0,\\
\left|\phi\left(\frac{2 \pi}{\sigma n}(\iota+1)\right)\right| & \textnormal{if } \iota < 0.
\end{cases}
\end{equation*}
Similar as above, we get
\begin{align*}
E_\textnormal{t}(t_i) 
\,&\leq\, \|w\|_{1} \left(\frac{2\pi}{\sigma n}\right)^d \mathrm{e}^{\lambda (\pi / \sigma)^2 d / 2} \, \sum_{j = 1}^d {d \choose j} \left(\frac{1}{\sqrt{2 \pi \lambda}} \left(\frac{\sigma n}{2 \pi}\right)\right)^{d - j} \left(2 \sum\nolimits_{\iota \,\geq\, q} \left|\phi\left(\frac{2 \pi}{\sigma n} \iota \right)\right|\right)^j\\
\,&\leq\, \|w\|_{1} \, (2\pi\lambda)^{- d / 2} \, \mathrm{e}^{\lambda (\pi / \sigma)^2 d / 2} \, \sum_{j = 1}^d {d \choose j} \, 2^j \, \bigg(\underbrace{\sum\nolimits_{\iota \,\geq\, q} \mathrm{e}^{- \iota^2 / 2 \lambda}}_{\eqqcolon\, M} \bigg)^j.
\end{align*}
The proof of \citep[Thm.\,1.7]{potts} shows that
\begin{equation*}
M \,=\, M(\sigma, q) \,\leq\, \left(1 + \frac{\lambda}{q}\right) \mathrm{e}^{- q^2 / 2 \lambda}.
\end{equation*}
Due to the special choice of $\lambda$, and again with the binomial formula, we obtain
\begin{align}\label{eq:nfft_error_t}
E_\textnormal{t}(t_i) 
\,&\leq\, \|w\|_{1} \, (3^d - 1) \, (2 \pi \lambda)^{- d / 2} \left(1 + \frac{\lambda}{q}\right)^d \mathrm{e}^{\lambda (\pi / \sigma)^2 d / 2 - q^2 / 2 \lambda}\\\notag
\,&=\, \|w\|_{1} \, \frac{3^d - 1}{q^{d / 2}} \, \bigg(\sqrt{\frac{2 \sigma - 1}{2 \sigma}} + \frac{1}{2\pi} \sqrt{\frac{2 \sigma}{2 \sigma - 1}} \, \bigg)^d \mathrm{e}^{- q \pi \left(1 - \frac{1}{2\sigma}\left(1 + \frac{d}{2 \sigma - 1}\right)\right)}.
\end{align}
Inserting \eqref{eq:nfft_error_a} and \eqref{eq:nfft_error_t} into \eqref{eq:nfft_error_1} and taking the maximum over all $i \in \lbrace 0, \dots, m-1 \rbrace$ completes the proof.
\end{proof}


We note that the bound for the error $E_{\infty}$ in Theorem \ref{thm:nfft_error} does not directly depend upon $n$, but that $q$ must be chosen as $q < \sigma n$. By choosing the oversampling factor $\sigma$ large enough, we can see that the error decays not less than \emph{exponentially} w.\,r.\,t.\ the cutoff parameter $q$.

To close this section, we review how the \emph{inverse} transform (iNFFT), \textit{i.\,e.}, the fast algorithm to compute the Fourier coefficients $\omega_k$, $k \in I^d_n$, of a $d$-dimensional trigonometric polynomial $p$ of degree at most $n$ from $m$ given scattered data $p(t_0), \dots, p(t_{m-1})$, can be constructed from the NFFT and its adjoint. For this, there are different possibilities, of which we only consider one particular here. For a more extended discussion and potential further developments of the method discussed here, we refer to \citep[Chap.\,5]{kunis} (see also \citep[Sect.\,1.7]{potts}).

Let $N \in \mathbb{C}^{m \times n^d}$ and $f \coloneqq [p(t_i)]_{i = 0, \dots, m - 1} \in \mathbb{C}^m$ be as in \eqref{eq:nfft_matrix}. The aim is to find a Fourier vector $\tilde{f} \in \mathbb{C}^{n^d}$ that solves the linear system
\begin{equation}\label{eq:infft}
N \tilde{f} \,=\, f.
\end{equation}
Under the above assumptions, this system has at least one solution. To compute such, we state \eqref{eq:infft} as a least-squares problem: Determine $\tilde{f}$ so that
\begin{equation}\label{eq:infft_least_squares}
\|f - N \tilde{f}\|_2^2 \,\leq\, \|f - Ng\|_2^2 \quad\quad \forall g \in \mathbb{C}^{n^d}.
\end{equation}

We distinguish between three different cases. If the number of sample points is larger than the number of potentially non-zero Fourier coefficients ($n^d < m$), then the solutions $\tilde{f}$ of \eqref{eq:infft_least_squares} are obtained by solving the \emph{normal equations of first kind} of the over-determined system \eqref{eq:infft},
\begin{equation}\label{eq:nfft_normal_equations}
N^\textnormal{H} N \tilde{f} \,=\, N^{\textnormal{H}}  f.
\end{equation}
If we assume in addition that the columns of $N$ are linearly independent (\textit{i.\,e.}, $\textnormal{rank}(N) = n^d$), then the solution $\tilde{f}$ is unique. The \emph{conjugate-gradient normal-equation residual} (CGNR) method lends itself well to the numerical solution of \eqref{eq:nfft_normal_equations} (see \citep[Alg.\,10.4.1]{golub_van_loan}). The advantage of this method in this context is that it is based on multiplications by the matrices $N$ and $N^\textnormal{H}$, a task for which we have the NFFT and its adjoint as fast algorithms.

If, on the other hand, the number of points is smaller than the number of Fourier coefficients ($m < n^d$), then the least-squares problem \eqref{eq:infft_least_squares} is reformulated as an optimization problem: Determine $\tilde{f}$ so that
\begin{equation*}
\|f - N\tilde{f}\|_2^2 \,\leq\, \|f - Ng\|_2^2 \quad \forall g \in \mathbb{C}^{n^d}, \quad\quad \|\tilde{f}\|_2 = \min.
\end{equation*}
With this additional condition, the solution $\tilde{f}$ is uniquely determined independently of the rank of $N$. If we assume that the rows of $N$ are linearly independent (\textit{i.\,e.}, $\textnormal{rank}(N) = m$), then we obtain $\tilde{f}$ by solving the \emph{normal equations of second kind} of the under-determined system \eqref{eq:infft},
\begin{equation}\label{eq:nfft_normal_equations_2}
N N^\textnormal{H} g \,=\, f, \quad\quad \tilde{f} \,=\, N^\textnormal{H} g.
\end{equation}
The \emph{conjugate-gradient normal-equation error} (CGNE) method is well suited for the numerical solution of \eqref{eq:nfft_normal_equations_2} (see \citep[Alg.\,10.4.2]{golub_van_loan})). Here again we can use the NFFT and its adjoint for the required multiplications with $N$ and $N^\textnormal{H}$, respectively.

Finally, if the number of sample points is exactly the same as the number of Fourier coefficients ($m = n^d$), then we can also apply the CGNR method to the normal equations \eqref{eq:nfft_normal_equations}; with the NFFT and its adjoint, we obtain a fast algorithm here as well.

\section{Discrete cosine transform (DCT)}
\label{sec:dct}
One cannot really speak of \emph{the} discrete cosine transform, for there are multiple classes of underlying discrete transforms to be distinguished (cf.\ Rem.\,\ref{rem:dct}). Here, we consider a particular type and call this \emph{the} DCT.
\begin{definition}[DCT]\label{def:dct}
Let $n \in \mathbb{N}$. We set
\begin{align*}
D_n \,&\coloneqq\,
\textnormal{diag}
\begin{bmatrix}
1 / \sqrt{n}, \,\sqrt{2 / n}, \,\dots, \,\sqrt{2 / n}\,
\end{bmatrix} \,\in\, \mathbb{R}^{n \times n},\\
C_n \,&\coloneqq\,
\begin{bmatrix}
\cos (i \omega_j)
\end{bmatrix}_{i, j = 0, \dots, n - 1} \,\in\, \mathbb{R}^{n \times n},
\end{align*}
where $\omega_j \coloneqq (2j + 1) \pi / 2n$. The linear mapping $\tilde{C}_n \coloneqq D_n C_n : \mathbb{C}^n \to \mathbb{C}^n$ is called the \emph{discrete cosine transform}. 
\end{definition}

\begin{remark}\label{rem:dct}
In the literature, four different variants DCT I to IV are typically distinguished from each other (see, \textit{e.\,g.}, \citep[Sect.\,2]{plonka_tasche}). These are the four versions established in practice of eight theoretically possible \citep{strang}. The DCT in Definition \ref{def:dct} is closely related to the DCT II, which is often referred to as \emph{the} discrete cosine transform.
\end{remark}

For a given vector of length $n$, a direct multiplication by the matrix $\tilde{C}_n$ apparently requires $\mathcal{O}(n^2)$ steps. In the context of the discrete cosine transform, the fast algorithms for performing the DCT are commonly also abbreviated as DCT, instead of FCT for \emph{fast cosine transform}. Being closely related to the iFFT of length $2n$, these fast algorithms have an asymptotic complexity of $\mathcal{O}(n \log n)$ (see \citep{plonka_tasche} for examples of such fast DCTs). 

As a first important property of the DCT, we note without proof that it is an \emph{orthogonal} transform, which also answers the question regarding the inverse transform (iDCT):

\begin{lemma}\label{lem:idct}
It is\/ $\tilde{C}_n^{-1} = \tilde{C}_n^\textnormal{T}$.
\end{lemma}

For two vectors $x$ and $y$ of length $n$, Lemma \ref{lem:idct} implies that $\langle x,y \rangle_2 = \langle \tilde{C}_n x, \tilde{C}_n y \rangle_2$. The DCT is thus an \emph{isometric isomorphism} w.\,r.\,t.\ the standard Euclidean norm, \textit{i.\,e.}, $\|\tilde{C}_n x\|_2 = \|x\|_2$. In Section \ref{sec:error}, we need the following estimate for the $1$-norm:

\begin{lemma}\label{lem:dct_1_norm}
Let\/ $n \in \mathbb{N}$ and\/ $x \in \mathbb{C}^n$. Then\/ $\|\tilde{C}_n x\|_1 \leq \sqrt{n} \, \|x\|_1$.
\end{lemma}

\begin{proof}
From the isometry property of the DCT w.\,r.\,t.\ the $2$-norm, it follows with the Cau\-chy\--Schwarz inequality that
\begin{equation*}
\|\tilde{C}_n x\|_1 \,\leq\, \sqrt{n} \, \|\tilde{C}_n x\|_2 \,=\, \sqrt{n} \, \|x\|_2 \,\leq\, \sqrt{n} \, \|x\|_1.
\end{equation*}
\end{proof}

Another property of the DCT is of particular importance to us, for it makes working with (trigonometric) polynomials very easy. To see this, we first introduce the Chebyshev polynomials (of first kind),
\begin{equation}\label{eq:chebyshev}
T_k : [-1,1] \to \mathbb{R}, \quad\quad T_k(\cos\omega) \coloneqq \cos(k\omega) \quad\quad (k \in \mathbb{N}_0).
\end{equation}
From the well-known cosine addition theorem, it follows that $T_k$ is a polynomial of degree $k$. Hence, for fixed $n \in \mathbb{N}_0$, the first $n + 1$ Chebyshev polynomials $T_0, \dots, T_n$ constitute a basis of the polynomial space $\Pi_n([-1,1])$. The next lemma shows how to expand polynomials efficiently w.\,r.\,t.\ the Chebyshev basis with the DCT.

\begin{lemma}[{cf.\ \citep[Sect.\,3]{kunis_potts}}]\label{lem:dct_chebyshev}
Let
\begin{equation*}
p \,=\, \sum_{k = 0}^{n - 1} \alpha_k \, T_k
\end{equation*}
be a polynomial on\/ $[-1,1]$ of degree at most\/ $n - 1$, and\/ $\omega_j = (2j + 1) \pi / 2n$ for\/ $j = 0, \dots, n - 1$. Then it is
\begin{equation*}
\begin{bmatrix}
\alpha_k
\end{bmatrix}_{k = 0, \dots, n - 1}
\,=\, D_n \tilde{C}_n  
\begin{bmatrix}
p(\cos\omega_j)
\end{bmatrix}_{j = 0, \dots, n - 1}.
\end{equation*}
\end{lemma}


\section{Fast Legendre transform (FLT)}
\label{sec:flt}
The associated Legendre polynomial of degree $l \in \mathbb{N}_0$ and order $m \in \lbrace -l,\dots,l \rbrace$ is defined as
\begin{equation*}
P_{lm} : [-1,1] \to \mathbb{R}, \quad\quad P_{lm}(\xi) \,\coloneqq\, \frac{(-1)^m}{2^l l!} \, (1 - \xi^2)^{m/2} \, \frac{\mathrm{d}^{l+m}}{\mathrm{d}\xi^{l+m}} (\xi^2 - 1)^l.
\end{equation*}
Note that the associated Legendre polynomials are, in fact, only polynomials for even order $m$. They are sometimes also referred to as associated Legendre \emph{functions}. The associated Legendre polynomial $P_{lm}$ constitutes the polar part of the spherical harmonic $Y_{lm}$. 
%

For $n \in \mathbb{N}$, we set $\vartheta_j \coloneqq (2j+1) \pi / 4n$, $j = 0, \dots, 2n - 1$, and define the Legendre matrices
\begin{equation}\label{eq:legendre_matrix}
L_m \,=\, L_{m,n} 
\,\coloneqq\,
\begin{bmatrix}
P_{lm}(\cos \vartheta_j)
\end{bmatrix}_{\subalign{l &= |m|, \dots, n-1\\ j &= 0, \dots, 2n-1}} \,\in\, \mathbb{R}^{(n - |m|) \times 2n}, \quad\quad |m| < n.
\end{equation}

\begin{definition}[DLT]\label{def:flt}
Let $n \in \mathbb{N}$. The sequence of linear mappings $L_m : \mathbb{C}^{2n} \to \mathbb{C}^{n-|m|}$, $m = 1 - n, \dots, n-1$, is called \emph{discrete Legendre transform} (DLT).
\end{definition}

In addition to $n$, let data $x_m \in \mathbb{C}^{2n}$, $|m| < n$, be given. When precomputing each required matrix $L_m$, $\mathcal{O}(n^2)$ steps are necessary for directly computing each matrix-vector product $L_m x_m$. This results in a naive algorithm for performing the DLT with an arithmetic and storage complexity of $\mathcal{O}(n^3)$. In recent years, many FLTs with a lower complexity have been proposed (see, for example, \citep{driscoll_healy, healy_et_al, kunis_potts}). This is due to the fact that the FLT constitutes an integral part of the fast spherical Fourier transform. Pursuing the well-known \emph{divide-and-conquer} strategy (see, \textit{e.g.}, \citep[Sect.\,2.3.1]{cormen_et_al}), \citeauthor{healy_et_al}\ develop FLTs with an arithmetic and storage complexity of $\mathcal{O}(n^2 \log^2 n)$; see \citep[Thm.\,3]{healy_et_al} and note that the \textit{precomputed data structure} is required solely for the FLT. \citet[Sect.\,4]{kunis_potts} offer FLTs with a complexity of $\mathcal{O}(n^2 \log^2 n)$ as well; these authors even also carry out a stabilization for large problem sizes. These elaborate FLTs, however, are more of theoretical interest to us: for the three-dimensional problems considered in this work, the DCT-based \emph{semi-naive} FLT and its adjoint of \citeauthor{healy_et_al}\ are suitable choices (cf.\ \citep[Sect.\,6]{healy_et_al}). The semi-naive FLT and its adjoint\,--\,which correspond to a factorization of the matrices $L_m$ and $L_m^\textnormal{T}$, respectively, see \citep[Thm.\,2.2.14 \& Cor.\,2.2.15]{wuelker}\,--\,have an asymptotic and storage complexity of $\mathcal{O}(n^3)$, \textit{i.e.}, they are no ``truly fast'' algorithms. The number of required computation steps, however, is significantly reduced in this variant. As an alternative to the semi-naive FLTs, one could also use the Clenshaw-Smith algorithm (Sect.\,\ref{sec:clenshaw}) or an FDPT and its respective adjoint to obtain a fast FLT and adjoint. This is due to the fact that the associated Legendre polynomials satisfy the three-term recurrence relation
\begin{equation}\label{eq:Legendre_recurrence}
(l + 1 - m) \, P_{l+1,m}(\xi) \,=\, (2l + 1) \, \xi \, P_{lm}(\xi) \,-\, (l + m) \, P_{l-1,m}(\xi), \quad\quad |m| \leq l \in \mathbb{N};
\end{equation}
cf.\ Remark \ref{rem:flt_clenshaw}.

\section{Clenshaw-Smith algorithm\,/\,fast discrete polynomial transform (FDPT)}
\label{sec:clenshaw}
Consider a function system $\lbrace f_k : k \in \mathbb{N}_0 \rbrace$ satisfying a three-term recurrence relation
\begin{equation}\label{eq:clenshaw_recursion}
f_{k + 1}(\xi) \,=\, \alpha_k(\xi) \, f_k(\xi) \,+\, \beta_k(\xi) \, f_{k-1}(\xi), \quad\quad k \in \mathbb{N}.
\end{equation}
Here we assume that the coefficients functions $\alpha_k$ and $\beta_k$ can be evaluated in $\mathcal{O}(1)$ steps. We are looking for an efficient method to evaluate the sums
\begin{equation}\label{eq:clenshaw_sum}
S_j \,\coloneqq\, \sum_{k = 0}^{n - 1} \gamma_k \, f_k(\xi_j), \quad\quad j = 0, \dots, m - 1,
\end{equation}
with given data $[\gamma_0, \dots, \gamma_{n-1}]$ at given points $[\xi_0, \dots, \xi_{m-1}]$ ($m,n \in \mathbb{N}$). This problem can be stated in matrix-vector notation as
\begin{equation}\label{eq:clenshaw_matrix}
\begin{bmatrix}
S_0\\
\vdots\\
S_{m-1} 
\end{bmatrix}
\,=\,
\underbrace{
\begin{bmatrix}
f_0(\xi_0) & \cdots & f_{n-1}(\xi_0)\\
\vdots & & \vdots\\
f_0(\xi_{m-1}) & \cdots & f_{n-1}(\xi_{m-1})\\
\end{bmatrix}
}_{\eqqcolon\, A}
\cdot
\begin{bmatrix}
\gamma_0\\
\vdots\\
\gamma_{n-1}
\end{bmatrix}.
\end{equation}

If one precomputes the matrix $A = A(\xi_0, \dots, \xi_{m-1}) \in \mathbb{C}^{m \times n}$ using the three-term recurrence relation \eqref{eq:clenshaw_recursion}, then an evaluation of \eqref{eq:clenshaw_matrix} with an arithmetic and storage complexity of $\mathcal{O}(mn)$ is possible. The Clenshaw-Smith algorithm \citep{clenshaw,smith}, which was first introduced for the Chebyshev polynomials \eqref{eq:chebyshev}, achieves this with a lower storage complexity of only $\mathcal{O}(m)$. The Clenshaw-Smith algorithm corresponds to a factorization of the matrix $A$ (cf.\ \citep[Thm.\,2.2.21]{wuelker}). Hence, one also has an adjoint Clenshaw-Smith algorithm, which allows for given points $[\xi_0, \dots, \xi_{m-1}]$ and corresponding data $[\gamma_0, \dots, \gamma_{m-1}]$ a computation of the sums
\begin{equation*}
\sum_{j = 0}^{m - 1} \gamma_j \, f_k(\xi_j), \quad\quad k = 0, \dots, n - 1,
\end{equation*}
with the same arithmetic and storage complexity as that of the Clenshaw-Smith algorithm ($m,n \in \mathbb{N}$; see also \citep[Cor.\,2.2.23]{wuelker}).

As indicated above, an alternative to the adjoint Clenshaw-Smith algorithm is the FDPT of \citet{driscoll_healy_rockmore}. In the case $m = n$, this class of fast algorithms has an arithmetic complexity of only $\mathcal{O}(n \log^2 n)$. The FDPT of \citeauthor{driscoll_healy_rockmore}\ corresponds to a factorization of the Hermitean transpose of the matrix $A$ in \eqref{eq:clenshaw_matrix} in which matrices of Toeplitz structure arise. This particular structure allows for a fast computation of the corresponding matrix-vector products using the FFT and its inverse (see \citep[Sect.\,4.2.4]{van_loan}). The storage complexity of the FDPT of \citeauthor{driscoll_healy_rockmore}, on the other hand, is $\mathcal{O}(n \log n)$, which slightly higher than that of the Clenshaw-Smith algorithm (cf.\ \citep[Sect.\,2.2, Rem.\,1]{driscoll_healy_rockmore}). 
A different, DCT-based FDPT was presented by \citep{potts_steidl_tasche_2}.
In this work, however, the FDPT is more of theoretical interest, as for the considered problem sizes, no significant advantage over the Clenshaw-Smith algorithm is to be expected.

\begin{remark}\label{rem:flt_clenshaw}
Since the associated Legendre polynomials satisfy the three-term recurrence relation \eqref{eq:Legendre_recurrence}, which is of the form \eqref{eq:clenshaw_recursion}, using the adjoint Clenshaw-Smith algorithm results in an FLT with an arithmetic complexity of $\mathcal{O}(n^3)$ and a storage complexity of $\mathcal{O}(n)$. If one uses instead of the Clenshaw-Smith algorithm an FDPT, then one obtains an FLT with an arithmetic complexity of only $\mathcal{O}(n^2 \log^2 n)$. When using the FDPT of \citeauthor{driscoll_healy_rockmore}, the storage complexity of such FLT is $\mathcal{O}(n^2 \log n)$ or even only $\mathcal{O}(n \log n)$ (cf.\ \citep[Sect.\,2.2, Rems.\ 1 \& 2]{driscoll_healy_rockmore}).
\end{remark}

\section{Derivation of the fast algorithms}
\label{sec:nfsglft}
Let the SGL Fourier coefficients $\hat{f}_{nlm}$ of a bandlimited function $f \in H$ with bandwidth $B \geq 2$, as well as points $x_i = [r_i, \vartheta_i, \varphi_i] \in \mathbb{R}^3$, $i = 0, \dots, M-1$, be given. Further, choose $\rho > 0$ such that $r_i \leq \rho$ holds for all $i \in \lbrace 0, \dots, M-1\rbrace$. We introduce the auxiliary function
\begin{equation*}
\gamma(r) \,=\, \gamma(\rho; r) \,\coloneqq\, \frac{2r - \rho}{\rho}, \quad\quad r \in [0, \rho].
\end{equation*}
The function $\gamma$ is a polynomial of degree one, mapping the interval $[0, \rho]$ bijectively onto the interval $[-1,1]$. We denote its inverse by $\gamma^{-1}$.

Now consider
\begin{align*}
f(r, \vartheta, \varphi) 
\,&=\, \sum\nolimits_{|m| \leq l < n \leq B} \hat{f}_{nlm} \, H_{nlm}(r, \vartheta, \varphi)\\
\,&=\, \sum\nolimits_{|m| \leq l < B} \underbrace{\left(\, \sum_{n = l+1}^{B} \hat{f}_{nlm} \, N_{nl} \, R_{nl}(r)\right)}_{\eqqcolon\, g_{lm}(r)} Y_{lm}(\vartheta,\varphi).\notag
\end{align*}
The functions $g_{lm}(r)$ are polynomials on $[0, \rho]$ of degree at most $2B - 2$. Using the auxiliary function $\gamma$, we rewrite them as
\begin{equation*}
g_{lm} \,=\, g_{lm} \circ \gamma^{-1} \circ \gamma.
\end{equation*}
The functions $\tilde{g}_{lm} \coloneqq g_{lm} \circ \gamma^{-1}$ are thus polynomials on $[-1,1]$ of degree at most $2B - 2$. With the DCT, we can expand these polynomials efficiently w.\,r.\,t.\ to the Chebyshev polynomials \eqref{eq:chebyshev}. To this end, we need to compute for $j = 0, \dots, 2B-1$ the values
\begin{equation*}
\tilde{g}_{lm}(\cos\omega_j) \,=\, g_{lm}(\gamma^{-1}(\cos\omega_j)) \,=\, g_{lm}\left(\frac{\rho}{2}(1 + \cos\omega_j)\right)
\end{equation*}
with $\omega_j \coloneqq (2j + 1) \pi / 4B$. The Clenshaw-Smith algorithm can achieve this for all pairs $[l, m]$, $|m| \leq l < B$, in $\mathcal{O}(B^2)$ steps each. This results in a complexity of $\mathcal{O}(B^4)$ for this first step. With the DCT, we can now compute for each pair $[l, m]$ the expansion coefficients 
\begin{equation}\label{eq:nfsglft_a_alpha}
\begin{bmatrix}
\alpha_{\kappa lm}
\end{bmatrix}_{\kappa = 0, \dots, 2B-1}
\,=\, D_{2B} \, \tilde{C}_{2B}
\begin{bmatrix}
\tilde{g}_{lm}(\cos\omega_j)
\end{bmatrix}_{j = 0, \dots, 2B-1}.
\end{equation}
This second step has a total complexity of $\mathcal{O}(B^3 \log B)$. By Lemma \ref{lem:dct_chebyshev}, it is
\begin{align*}
(\tilde{g}_{lm} \circ \cos)(\omega)
\,&=\, \sum_{\kappa = 0}^{2B-1} \alpha_{\kappa lm} \, (T_\kappa \circ \cos)(\omega)
\,=\, \sum_{\kappa = 0}^{2B-1} \alpha_{\kappa lm} \, \cos (\kappa \omega)
\,=\, \sum_{\kappa = 0}^{2B-1} \frac{\alpha_{\kappa lm}}{2} \big(\mathrm{e}^{\mathrm{i} \kappa \omega} + \mathrm{e}^{-\mathrm{i} \kappa \omega}\big).
\end{align*}
This results in (cf.\ Def.\,\ref{def:nfft_grid})
\begin{equation*}
g_{lm}(r_i) \,=\, \sum\nolimits_{\kappa \in I_{4B}^1} \beta_{\kappa lm} \, \mathrm{e}^{\mathrm{i} \kappa \arccos \gamma(r_i)}, \quad\quad i = 0, \dots, M-1,
\end{equation*}
wherein
\begin{equation}\label{eq:nfsglft_a_beta}
\beta_{\kappa lm} \,\coloneqq\,
\begin{cases}
\alpha_{0,l,m} & \textnormal{for } \kappa = 0,\\
\alpha_{|\kappa|,l,m}/2 & \textnormal{for } 0 < |\kappa| < 2B,\\
0 & \textnormal{for } \kappa = - 2B.
\end{cases}
\end{equation}
The complete radial subtransform described above has an arithmetic complexity of $\mathcal{O}(B^4)$, while the storage complexity is $\mathcal{O}(B^3)$, as can be checked easily. When using instead of the Clenshaw-Smith algorithm an adjoint FDPT, the arithmetic complexity reduces to $\mathcal{O}(B^3 \log^2 B)$, while the storage complexity remains the same.

We now consider for fixed $i \in \lbrace 0, \dots, M-1 \rbrace$ the spherical polynomial
\begin{align*}
f(r_i, \vartheta, \varphi) 
\,&=\, \sum\nolimits_{|m| \leq l < B} g_{lm}(r_i) \, Y_{lm}(\vartheta, \varphi)\\
\,&=\, \sum\nolimits_{\kappa_0 \in I_{4B}^1} \sum\nolimits_{|m| < B} \underbrace{\Bigg(\sum_{l = |m|}^{B-1} \beta_{\kappa_{0}, l,m} \, Q_{lm} \, P_{lm}(\cos\vartheta)\Bigg)}_{\eqqcolon\, h_{\kappa_0, m}(\cos\vartheta)} \mathrm{e}^{\mathrm{i} (\kappa_0 \arccos \gamma(r_i) + m \varphi)},
\end{align*}
where $Q_{lm} = \sqrt{\frac{2l+1}{4\pi}\frac{(l-m)!}{(l+m)!}}$ is the normalization constant of the spherical harmonic $Y_{lm}$.
For even $m$, the functions $h_{\kappa_{0}, m}$ are polynomials on $[-1,1]$ of degree as most $B-1$. It follows that
\begin{equation*}
h_{\kappa_0, m} \,=\, \sum_{\kappa = 0}^{2B-1} \varepsilon_{\kappa_0, \kappa, m} \, T_\kappa
\end{equation*}
with the expansion coefficients
\begin{equation}\label{eq:nfsglft_a_epsilon_even}
\begin{bmatrix}
\varepsilon_{\kappa_0,\kappa, m}
\end{bmatrix}_{\kappa = 0, \dots, 2B-1}
\,=\, D_{2B} \, \tilde{C}_{2B}
\begin{bmatrix}
h_{\kappa_0, m}(\cos\omega_j)
\end{bmatrix}_{j = 0, \dots, 2B-1}.
\end{equation}
As above, we get
\begin{equation*}
h_{\kappa_0, m}(\cos\vartheta_i) \,=\, \sum\nolimits_{\kappa \in I_{4B}^1} \zeta_{\kappa_0, \kappa, m} \, \mathrm{e}^{\mathrm{i} \kappa \vartheta_i},
\end{equation*}
wherein
\begin{equation}\label{eq:nfsglft_a_zeta_even}
\zeta_{\kappa_0, \kappa, m} \,\coloneqq\,
\begin{cases}
\varepsilon_{\kappa_0, 0, m} &\textnormal{for } \kappa = 0,\\
\varepsilon_{\kappa_0, |\kappa|, m} / 2 &\textnormal{for } 0 < |\kappa| < 2B,\\
0 &\textnormal{for } \kappa = - 2B.
\end{cases}
\end{equation}
In the case when $m$ is odd, the above approach is successful as well. In this case, $(1 - \xi^2)^{-1/2} h_{\kappa_0, m}(\xi)$ are polynomials on $[-1,1]$ of degree at most $B-2$. It follows that
\begin{equation*}
h_{\kappa_0, m}(\cos\omega) \,=\, \sum_{\kappa = 0}^{2B-2} \varepsilon_{\kappa_0, \kappa, m} \, \sin(\omega) \, T_\kappa(\cos\omega)
\end{equation*}
with the expansion coefficients
\begin{equation}\label{eq:nfsglft_a_epsilon_odd}
\begin{bmatrix}
\varepsilon_{\kappa_0, \kappa, m}
\end{bmatrix}_{\kappa = 0, \dots, 2B-1}
\,=\, D_{2B} \, \tilde{C}_{2B} \, V_B
\begin{bmatrix}
h_{\kappa_0, m}(\cos\omega_j)
\end{bmatrix}_{j = 0, \dots, 2B-1},
\end{equation}
where $V_B \coloneqq \textnormal{diag}[(\sin((2j+1)\pi/4n))^{-1}]_{j = 0, \dots, 2B-1}$ is an auxiliary matrix. Since
\begin{equation*}
\sin(\omega) \, T_\kappa(\cos\omega) \,=\, \sin(\omega) \cos(\kappa\omega) \,=\, \frac{1}{4\mathrm{i}}\big(\mathrm{e}^{\mathrm{i} (\kappa + 1) \omega} - \mathrm{e}^{- \mathrm{i} (\kappa + 1) \omega} - \mathrm{e}^{\mathrm{i} (\kappa - 1) \omega} + \mathrm{e}^{- \mathrm{i} (\kappa - 1) \omega}\big),
\end{equation*}
it is
\begin{equation*}
h_{\kappa_0, m}(\cos\vartheta_i) \,=\, \sum\nolimits_{\kappa \in I_{4B}^1} \zeta_{\kappa_0, \kappa, m} \, \mathrm{e}^{\mathrm{i} \kappa \vartheta_i},
\end{equation*}
wherein
\begin{equation}
\zeta_{\kappa_0, \kappa, m} \,\coloneqq\,
\frac{\textnormal{sgn} \, \kappa}{4 \mathrm{i}}
\begin{cases}\label{eq:nfsglft_a_zeta_odd}
0  & \textnormal{for } \kappa = 0,\\
2 \varepsilon_{\kappa_0, 0, m} - \varepsilon_{\kappa_0, 2, m} & \textnormal{for } |\kappa| = 1,\\
\varepsilon_{\kappa_0, |\kappa|-1, m} - \varepsilon_{\kappa_0, |\kappa| + 1, m} & \textnormal{for } 1 < |\kappa| < 2B - 2,\\
\varepsilon_{\kappa_0, |\kappa|-1, m} & \textnormal{for } 2B-2 \leq |\kappa| < 2B,\\
0  & \textnormal{for } \kappa = - 2B.
\end{cases}
\end{equation}

For each fixed $\kappa_0$, the function values $h_{\kappa_0, m}(\cos\omega_j)$ can be computed for all $m$ in a total of $\mathcal{O}(B^3)$ steps, using an adjoint FLT. For this, the adjoint FLT should be adapted such that in the case when $m$ is odd, the weighting by $V_B$ in \eqref{eq:nfsglft_a_epsilon_odd} is already included, in order to prevent stability issues from arising. The complexity of this step is $\mathcal{O}(B^4)$, or even only $\mathcal{O}(B^3 \log^2 B)$ when using an $\mathcal{O}(B^2 \log^2 B)$ adjoint FLT. Subsequently, for each fixed $\kappa_0$ and $m$, the coefficients $\varepsilon_{\kappa_0, \kappa, m}$ can be computed with the DCT. This step has a complexity of $\mathcal{O}(B^3 \log B)$. The above-described spherical subtransform thus has an overall arithmetic complexity of $\mathcal{O}(B^4)$, while the storage complexity amounts to $\mathcal{O}(B^3)$. When employing an $\mathcal{O}(B^2 \log^2 B)$ adjoint FLT, the arithmetic complexity is reduced to $\mathcal{O}(B^3 \log^2 B)$, while the storage complexity remains the same.

\begin{remark}\label{rem:nfsft}
The above spherical subtransform is essentially the \emph{non-equiangular fast spherical Fourier transform} (NFSFT) of \citet{kunis_potts}. As described above, spherical polynomials are there brought into the form of two-dimensional trigonometric polynomials, which can then be evaluated efficiently with the two-dimensional NFFT (see also \citep[Sect.\,3.3.1]{kunis}).
\end{remark}

We thus find that for $i = 0, \dots, M - 1$, it is
\begin{align}
f(x_i)
\,&=\, \sum\nolimits_{\kappa_0 \in I_{4B}^1} \sum\nolimits_{|m| < B} h_{\kappa_0, m}(\cos\vartheta_i) \, \mathrm{e}^{\mathrm{i} (\kappa_0 \arccos \gamma(r_i) + m \varphi_i)}\notag\\
\,&=\, \sum\nolimits_{\kappa_0 \in I_{4B}^1} \sum\nolimits_{\kappa_1 \in I_{4B}^1} \sum\nolimits_{|m| < B} \zeta_{\kappa_0, \kappa_1, m} \, \mathrm{e}^{\mathrm{i} (\kappa_0 \arccos \gamma(r_i) + \kappa_1 \vartheta_i + m \varphi_i)}\notag\\
\,&=\, \sum\nolimits_{k \in I_{4B}^3} \eta_k \, \mathrm{e}^{\mathrm{i} \langle k, \tilde{x}_i \rangle_2}\label{eq:nfsglft_nfft}
\end{align}
with the coefficients
\begin{equation*}
\eta_k \,\coloneqq\,
\begin{cases}
\zeta_{\kappa_0, \kappa_1, \kappa_2} & \textnormal{for } |\kappa_2 | < B,\\
0 & \textnormal{otherwise},
\end{cases} \quad\quad k \,\coloneqq\, [\kappa_0, \kappa_1, \kappa_2],
\end{equation*}
and with the transformed points
\begin{equation}\label{eq:x_tilde}
\tilde{x}_i \,\coloneqq\, [\arccos \gamma(r_i), \vartheta_i, \varphi_i] \,\in\, \mathbb{T}^3.
\end{equation}
In a last step, the right-hand side of \eqref{eq:nfsglft_nfft} can now be evaluated for all $i = 0, \dots, M - 1$ in a total of $\mathcal{O}((\sigma B)^3 \log (\sigma B) + q^3 M)$ steps, using the three-dimensional NFFT; 
here we let the oversampling factor $\sigma$ as well as the cutoff parameter $q < \sigma B$ be variable for now (cf.\ Sect.\,\ref{sec:nfft}). 

In summary, we have derived a class of NFSGLFTs with an arithmetic complexity of $\mathcal{O}(B^4 + (\sigma B)^3 \log (\sigma B) + q^3 M)$ or even only $\mathcal{O}(B^3 \log^2 B + (\sigma B)^3 \log (\sigma B) + q^3 M)$ and a storage complexity of $\mathcal{O}((\sigma B)^3)$. The role of the oversampling factor $\sigma$ and the cutoff parameter $q$ is elaborated in the next section.

We shall render the above class of NFSGLFTs as a factorization of the transformation matrix $\Lambda$ in Definition \ref{def:ndsglft}. To this end, we define the auxiliary matrices
\begin{equation*}
R_l \,=\, R_{l,B}(\rho) \,\coloneqq\,
\begin{bmatrix}
N_{nl} R_{nl}(\rho (1 + \cos\omega_j) / 2)
\end{bmatrix}_{\subalign{j &= 0, \dots, 2B-1 \\ n &= l+1, \dots, B}} \,\in\, \mathbb{R}^{2B \times (B-l)}, \quad\quad l < B,
\end{equation*}
associated with the radial subtransform. Let further $L_m = L_{m, B}$ be the Legendre matrices defined in \eqref{eq:legendre_matrix}. We introduce the permutation matrix
\begin{equation*}
S_B \,\coloneqq\,
\begin{bmatrix}
e_{\varkappa(\mu)}^\textnormal{T}
\end{bmatrix}_{\mu = 0, \dots, B(B+1)(2B+1)/6-1} \,\in\, \mathbb{R}^{B(B+1)(2B+1)/6 \times B(B+1)(2B+1)/6}
\end{equation*}
with the canonical unit (column) vectors $e_{\varkappa(\mu)}$ of length $B(B+1)(2B+1)/6$, where (cf.\ Eqs.\,\ref{eq:n_mu_l_mu_m_mu})
\begin{align*}
\varkappa(\mu) \,\coloneqq\, n(\mu) &+ \left(B + \left(\frac{2B-1}{2} - \frac{2l(\mu)-1}{3}\right)\Big(l(\mu)-1\Big) - 1\right)l(\mu) \\ 
&+ \Big(B-l(\mu)\Big)\Big(l(\mu)+m(\mu)\Big) - 1. \vphantom{\left(\frac{B}{2}\right)}
\end{align*}
The transposed matrix $S_B^\textnormal{T}$ resorts $n = 1, \dots, B$; $l = 0, \dots, n - 1$; $m = - l, \dots, l$, to $l = 0, \dots, B - 1$; $m = - l, \dots, l$; $n = l + 1, \dots, B$. Further, we employ the permutation matrix
\begin{equation*}
U_B \,\coloneqq\,
\begin{bmatrix}
e_{\varsigma(\psi)}^\textnormal{T}
\end{bmatrix}_{\psi = 0, \dots, 4B^3-1} \,\in\, \mathbb{R}^{4B^3 \times 4B^3}
\end{equation*}
with the canonical unit vectors $e_{\varsigma(\psi)}$ of length $4B^3$, where
\begin{equation*}
\varsigma(\psi) \,\coloneqq\, B^2 \, \kappa(\psi) + \frac{B^2 - B - m(\psi)^2 \, \textnormal{sgn} \, m(\psi)}{2} + \left(B - \frac{\textnormal{sgn} \, m(\psi)}{2} \right) m(\psi) + l(\psi),
\end{equation*}
wherein
\begin{equation*}
\kappa(\psi) \,\coloneqq\, \psi \textnormal{ mod } 4B, \vphantom{\frac{\psi}{2B}}\quad\quad
l(\psi) \,\coloneqq\, \bigg\lfloor \sqrt{\frac{\psi - \kappa(\psi)}{4B}} \bigg\rfloor,\quad\quad
m(\psi) \,\coloneqq\, \frac{\psi - \kappa(\psi)}{4B} - l(\psi)(l(\psi) + 1).
\end{equation*}
The transposed matrix $U_B^\textnormal{T}$ resorts $l = 1-B, \dots, B-1$; $m = -l, \dots, l$; $\kappa = -2B, \dots 2B-1$, to $\kappa = -2B, \dots, 2B-1$; $m = 1-B, \dots, B-1$; $l = |m|,\dots,B-1$. As a last permutation matrix, we introduce
\begin{equation*}
X_B \,\coloneqq\,
\begin{bmatrix}
e_{\tau(\iota)}^\textnormal{T}
\end{bmatrix}_{\iota = 0, \dots, 4B^3 - 1} \,\in\, \mathbb{R}^{4B(2B-1) \times 4B(2B-1)}
\end{equation*}
with the canonical unit vectors $e_{\tau(\iota)}$ of length $4B(2B-1)$, where
\begin{equation*}
\tau(\iota) \,\coloneqq\, B + (2B + \kappa(\iota))(2B - 1) + m(\iota) - 1,
\end{equation*}
wherein
\begin{equation*}
\kappa(\iota) \,\coloneqq\, (\iota \textnormal{ mod } 4B) - 2B,\quad\quad
m(\iota) \,\coloneqq\, \frac{\iota - \kappa(\iota) - 2B}{4B} - B + 1.
\end{equation*}
The transposed matrix $X_B^\textnormal{T}$ is for resorting $m = 1-B, \dots, B-1$; $\kappa = -2B, \dots, 2B - 1$, to $\kappa = -2B, \dots, 2B - 1$; $m = 1-B, \dots, B-1$.
In addition to the above permutation matrices, we employ the auxiliary matrices
\begin{equation*}
A_B \,\coloneqq\, \frac{1}{2}
\begin{bmatrix}
\begin{array}{c|ccc|c|ccc}
0 & & & & 2 & & &\\\hline
& & & 1 & & 1 & &\\
& & \textnormal{\reflectbox{$\ddots$}} & & & & \ddots &\\
& 1 & & & & & & 1
\end{array}
\end{bmatrix}^\textnormal{T} \in\, \mathbb{R}^{4B \times 2B},
\end{equation*}
as well as $W_m = W_{m,B} \coloneqq A_B$ for $m$ even and
\begin{equation*}
W_m 
\,\coloneqq\, \frac{1}{4 \mathrm{i}}
\begin{bmatrix}
\begin{array}{c|ccccc|c|ccccc}
0 & & & & & -2 & 0 & 2 & & &  &\\\hline
& & & & -1 & 0 & & 0 & 1 & & &\\
& & & -1 & 0 & 1 & & -1 & 0 & 1 & &\\
& & \textnormal{\reflectbox{$\ddots$}} & \textnormal{\reflectbox{$\ddots$}} & \textnormal{\reflectbox{$\ddots$}} & & & & \ddots & \ddots & \ddots &\\
& -1 & 0 & 1 & & & & & & -1 & 0 & 1\\\hline
0 & & & & & & & & & & &        
\end{array}
\end{bmatrix}^\textnormal{T} \in\, \mathbb{R}^{4B \times 2B}
\end{equation*}
for $m$ odd. With these matrices, the relation between \eqref{eq:nfsglft_a_alpha} and \eqref{eq:nfsglft_a_beta} can be written as
\begin{equation*}
\begin{bmatrix}
\beta_{\kappa lm}
\end{bmatrix}_{\kappa=-2B,\dots,2B-1}
\,=\, A_B
\begin{bmatrix}
\alpha_{\kappa lm}
\end{bmatrix}_{\kappa = 0,\dots,2B-1},
\end{equation*}
and the relation between \eqref{eq:nfsglft_a_epsilon_even}, or respectively \eqref{eq:nfsglft_a_epsilon_odd}, and \eqref{eq:nfsglft_a_zeta_even}, or respectively \eqref{eq:nfsglft_a_zeta_odd}, as
\begin{equation*}
\begin{bmatrix}
\zeta_{\kappa_0, \kappa, m}
\end{bmatrix}_{\kappa = -2B, \dots, 2B-1}
\,=\, W_m
\begin{bmatrix}
\varepsilon_{\kappa_0, \kappa, m}
\end{bmatrix}_{\kappa = 0, \dots, 2B-1}.
\end{equation*}
Both above relations represent a change from the Chebyshev to the monomial basis. As yet another auxiliary matrix, we define
\begin{equation*}
Z_B \,\coloneqq\,
\begin{bmatrix}\!
\begin{array}{c|c|c}
0_{2B-1, B+1} & \mathbbm{1}_{2B-1} & 0_{2B-1, B}
\end{array}\!
\end{bmatrix}
\,\in\, \mathbb{R}^{2B \times 4B}
\end{equation*}
with the zero matrices $0_{2B-1, B+1} \in \mathbb{R}^{(2B-1) \times (B+1)}$ and $0_{2B-1, B} \in \mathbb{R}^{(2B-1) \times B}$, and where here and in the following $\mathbbm{1}_n$ generally denotes the identity matrix of size $n \times n$.
The transposed matrix $Z_B^\textnormal{T}$ extends the range $\kappa = 0, \dots, B-1$ to $\kappa = -2B, \dots, 2B-1$ by zero padding. Finally, for the given points $x_0, \dots, x_{M-1}$, we introduce the special NDFT matrix (cf.\ Eq.\,\ref{eq:nfft_matrix})
\begin{equation*}
N \,=\, N(B; x_0, \dots, x_{M-1}) \,\coloneqq\,
\begin{bmatrix}
\mathrm{e}^{\mathrm{i} \langle k(\chi), \tilde{x}_i \rangle_2}
\end{bmatrix}_{\subalign{i &= 0, \dots, M-1\\ \chi &= 0, \dots, (4B)^3-1}} \in\, \mathbb{C}^{M \times (4B)^3}
\end{equation*}
with the transformed points $\tilde{x}_i$ defined in \eqref{eq:x_tilde}.
In combining all the above components, we can state the following main result:

\begin{theorem}\label{thm:nfsglft}
The matrix\/ $\Lambda$ in Definition \ref{def:ndsglft} possesses the factorization
\begin{align*}
\Lambda \,&=\,N \cdot \left\lbrace \mathbbm{1}_{(4B)^2} \otimes Z_B^\textnormal{T} \right\rbrace
\cdot
\left\lbrace \mathbbm{1}_{4B} \otimes X_B^\textnormal{T} \right\rbrace
\cdot 
\left\lbrace 
\vphantom{
\begin{bmatrix}
W_{1-B} \, D_{2B} \, \tilde{C}_{2B} & &\\
& \ddots &\\
& & W_{B-1} \, D_{2B} \, \tilde{C}_{2B}
\end{bmatrix}
}
\right.
\mathbbm{1}_{4B} \otimes
\overbrace{
\begin{bmatrix}
W_{1-B} \, D_{2B} \, \tilde{C}_{2B} & &\\
& \!\!\!\ddots\!\!\! &\\
& & W_{B-1} \, D_{2B} \, \tilde{C}_{2B}
\end{bmatrix}}^{\hspace*{-100pt}2B - 1 \textit{ blocks of size\/ $4B \times 2B$}\hspace*{-100pt}} 
\left. \vphantom{
\begin{bmatrix}
W_{1-B} \, D_{2B} \, \tilde{C}_{2B} & &\\
& \ddots &\\
& & W_{B-1} \, D_{2B} \, \tilde{C}_{2B}
\end{bmatrix}
}
\right\rbrace\\[4pt] 
&\times\, \left\lbrace 
\vphantom{
\begin{bmatrix}
\tilde{L}_{1-B}^\textnormal{T} & &\\
& \ddots &\\
& & \tilde{L}_{B-1}^\textnormal{T}
\end{bmatrix}
}
\right. 
\mathbbm{1}_{4B} \otimes
\underbrace{
\begin{bmatrix}
\tilde{L}_{1-B}^\textnormal{T} & &\\
& \!\!\!\ddots\!\!\! &\\
& & \tilde{L}_{B-1}^\textnormal{T}
\end{bmatrix}}_{\hspace*{-100pt}2B - 1 \textit{ blocks (see below)}\hspace*{-100pt}} 
\left. 
\vphantom{
\begin{bmatrix}
\tilde{L}_{1-B}^\textnormal{T} & &\\
& \ddots & \\
& & \tilde{L}_{B-1}^\textnormal{T}
\end{bmatrix}
}
\right\rbrace
\cdot U_B^\textnormal{T}
\cdot 
\underbrace{
\left\lbrace \mathbbm{1}_{B^2} \otimes \left(A_B \, D_{2B} \, \tilde{C}_{2B}\right) \right\rbrace 
}_{\hspace*{-100pt}B^2 \textit{ blocks of size\/ } 4B \times 2B \hspace*{-100pt}}
\cdot
\underbrace{
\begin{bmatrix}
\tilde{R}_{0} & &\\
& \!\!\!\ddots\!\!\! &\\
& & \tilde{R}_{B-1}
\end{bmatrix}}_{\hspace*{-100pt}B \textit{ blocks (see below)}\hspace*{-100pt}}
\cdot \, S_B^\textnormal{T}
\end{align*}
with
\begin{equation*}
\tilde{L}_m^\textnormal{T} \,\coloneqq\,
\begin{Bmatrix}
\begin{aligned}
\mathbbm{1}_{2B} &\textit{ for\/ } m \textit{ even}\\
V_{B}  &\textit{ for\/ } m \textit{ odd}
\end{aligned}
\end{Bmatrix} 
\cdot L_m^\textnormal{T} \cdot 
\textnormal{diag}
\begin{bmatrix}
Q_{lm}
\end{bmatrix}_{l = |m|, \dots, B-1} \,\in\, \mathbb{R}^{2B \times (B-|m|)}, \quad\quad |m| < B,
\end{equation*}
and the block-diagonal matrices
\begin{equation*}
\tilde{R}_l \,\coloneqq\, 
\underbrace{\mathbbm{1}_{2l+1} \otimes R_l}_{\hspace*{-100pt}2l+1 \textit{ blocks of size\/ } 2B \times (B-l)\hspace*{-100pt}}, \quad\quad l < B.
\end{equation*}
\end{theorem}

The matrices $L_m^\textnormal{T}$ and $R_l$ can now be factorized themselves, as mentioned in Sections \ref{sec:flt} and \ref{sec:clenshaw}; when using the adjoint semi-naive FLT and for the radial part the Clenshaw-Smith algorithm, the factorization of these matrices is given by \citep[Cor.\,2.2.15]{wuelker} and \citep[Thm.\,2.2.21]{wuelker}, respectively.
Reverting the order of the factors and conjugate-transposing each factor while taking into account the laws of the Kronecker product, we get with Lemma \ref{lem:idct} as a direct consequence of Theorem \ref{thm:nfsglft} the following second main result of this section. It shows that we also have an adjoint NFSGFLT with the same arithmetic and storage complexity.

\begin{corollary}
The matrix\/ $\Lambda^\textnormal{H}$ can be factorized as
\begin{align*}
\Lambda^\textnormal{H} \,&=\,
S_B \cdot
\begin{bmatrix}
\tilde{R}_0^\textnormal{T} & &\\
& \!\!\!\ddots\!\!\! &\\
& & \tilde{R}_{B-1}^\textnormal{T}
\end{bmatrix}
\cdot \left\lbrace \mathbbm{1}_{B^2} \otimes \left(\tilde{C}_{2B}^{-1} \, D_{2B}^\textnormal{T} \, A_{B}^\textnormal{T}\right) \right\rbrace
\cdot U_B
\cdot \left\lbrace \mathbbm{1}_{4B} \otimes
\begin{bmatrix}
\tilde{L}_{1-B} & &\\
& \!\!\!\ddots\!\!\! &\\
& & \tilde{L}_{B-1}
\end{bmatrix} \right\rbrace\\[4pt]
&\times\,
\left\lbrace \mathbbm{1}_{4B} \otimes
\begin{bmatrix}
\tilde{C}_{2B}^{-1} \, D_{2B}^\textnormal{T} \, W_{1-B}^\textnormal{H} & &\\
& \!\!\!\ddots\!\!\! &\\
& & \tilde{C}_{2B}^{-1} \, D_{2B}^\textnormal{T} \, W_{B-1}^\textnormal{H} \,
\end{bmatrix} \right\rbrace
\cdot 
\left\lbrace \mathbbm{1}_{4B} \otimes X_B \right\rbrace \cdot \left\lbrace \mathbbm{1}_{(4B)^2} \otimes Z_B \right\rbrace
\cdot N^\textnormal{H}.
\end{align*}
\end{corollary}

The matrices $L_m$ and $R_l^\textnormal{T}$ contained here can now be factorized themselves as well, see \citep[Thm.\,2.2.14]{wuelker} for the semi-naive FLT and \citep[Cor.\,2.2.23]{wuelker} for the adjoint Clenshaw-Smith algorithm. In the same manner as explained for the NFFT at the end of Section \ref{sec:nfft}, the NFSGLFT and its adjoint can be employed for an iterative inverse NFSGLFT, \textit{i.\,e.}, the fast CG algorithm for computing the SGL Fourier coefficients $\hat{f}_{nlm}$ of a bandlimited function $f$ from given scattered data $f(x_i)$.

\section{Error estimate}
\label{sec:error}
By careful consideration it becomes apparent that the only approximating part of the N\-F\-SGL\-FT derived above is the final NFFT; the matrix factorization in Theorem \ref{thm:nfsglft}, on the other hand, is exact. With Theorems \ref{thm:nfft_error} and \ref{thm:nfsglft}, we can thus derive an estimate for the maximum absolute error of the NFSGLFT. The latter is defined as
\begin{equation}\label{eq:nfsglft_error_def}
E_\infty \,=\, E_\infty(B; \sigma, q; \rho; x_0, \dots, x_{M-1}; \hat{f}) \,\coloneqq\, \max_{i \in \lbrace 0, \dots, M-1\rbrace} |f(x_i) - \tilde{f}(x_i)|
\end{equation}
with the given SGL Fourier coefficients $\hat{f} = [\hat{f}_{n(\mu),l(\mu),m(\mu)}]_{\mu = 0, \dots, B(B+1)(2B+1)/6-1}$ and the output result $[\tilde{f}(x_i)]_{i = 0, \dots, M-1}$ of the NFSGLFT. Here, $\sigma$ is the oversampling factor of the NFFT and $q$ its cutoff parameter.

\begin{theorem}\label{thm:nfsglft_error}
Using the NFFT of Section \ref{sec:nfft} with the Gaussian ansatz function and\/ $\sigma \geq 2$, the maximum absolute error of the NFSGLFT of Section \ref{sec:nfsglft} is bounded by
\begin{equation*}
E_\infty \,\lesssim\, B^{7/2} \, a_{\rho, B} \, \exp\bigg(b_{\rho, B} \left(B + \frac{1}{2}\right)^{1-1/\mathrm{e}} + \frac{\rho^2}{2} - \frac{q \pi}{2}\bigg) \, \|\hat{f}\|_1
\end{equation*}
with the coefficients
\begin{align*}
a_{\rho, B} \,\coloneqq\,
&\begin{cases}
1 & \textit{if\/ } \rho < 1 \textit{ and\/ } B \leq \Omega(\rho),\\
\rho^{-1} & \textit{otherwise},
\end{cases}\quad\quad
b_{\rho, B} \,\coloneqq\, \frac{\mathrm{e}^{2 / \mathrm{e}}}{2}
\begin{cases}
1 & \textit{if\/ } \rho < 1 \textit{ and\/ } B \leq \Omega(\rho),\\
\rho^{2 / \mathrm{e}} & \textit{otherwise},
\end{cases}
\end{align*}
wherein
\begin{equation*}
\Omega(\rho) \coloneqq \bigg(\frac{\mathrm{e}^{2 / \mathrm{e}}(\rho^{2 / \mathrm{e}} - 1)}{2 \ln\rho}\bigg)^{1 - 1/\mathrm{e}} - \frac{1}{2}.
\end{equation*}
\end{theorem}

For the proof of this error estimate, we require the following technical Lemmata \ref{lem:laguerre_bound}--\ref{lem:legendre_bound_odd}.

\begin{lemma}[{cf.\ \citep[Eq.\,22.14.13]{abramowitz_stegun}}]\label{lem:laguerre_bound}
For\/ $\alpha \geq 0$, we have that
\begin{equation*}
\big| L_n^{(\alpha)}(\xi) \big| \,\leq\, {n + \alpha \choose n} \, \mathrm{e}^{\xi / 2}, \quad\quad \xi \in [0,\infty).
\end{equation*}
\end{lemma}

\begin{lemma}\label{thm:nfsglft_error_lem_1}
For\/ $\xi \geq \zeta > 0$, it is
\begin{equation*}
\ln \frac{\Gamma(\xi)}{\Gamma(\zeta)} \,\leq\, (\xi - \zeta) \, \ln \xi.
\end{equation*}
\end{lemma}

\begin{proof}
Due to the log convexity of the gamma function on $(0,\infty)$ (see \citep[Cor.\,1.2.6]{andrews_askey_roy}), we first have that $\ln\Gamma(\xi) - \ln\Gamma(\zeta) \leq (\xi - \zeta) \psi(\xi)$ with the digamma function $\psi \coloneqq \mathrm{d} \ln \Gamma / \mathrm{d} \xi$. The Lemma now follows from the estimate $\psi(\xi) \leq \ln \xi$, $\xi \in (0,\infty)$ (cf.\ \citep[Eq.\,6.3.21]{abramowitz_stegun}).
\end{proof}

\begin{lemma}[{\citet[Thm.\,1.5]{batir}}]\label{thm:nfsglft_error_lem_2}
For\/ $\xi > 1$, we have that
\begin{equation*}
\sqrt{2 \mathrm{e}} \left(\frac{\xi - 1/2}{\mathrm{e}}\right)^{\xi-1/2} \,\leq\, \Gamma(\xi). 
\end{equation*}
\end{lemma}

\begin{lemma}\label{thm:nfsglft_error_lem_3}
For\/ $\xi > 0$, it is\/ $\ln \xi \leq \xi^{1 / \mathrm{e}}$.
\end{lemma}

\begin{proof}
The function $\xi^{1 / \mathrm{e}} - \ln \xi$ has the global maximum zero, attained at $\mathrm{e}^\mathrm{e}$.
\end{proof}

\begin{lemma}\label{lem:legendre_bound_even}
The associated Legendre polynomials can be estimated as
\begin{equation*}
\sqrt{\frac{(l-m)!}{(l+m)!}} \, |P_{lm}(\xi)| \,\leq\, 1.
\end{equation*}
\end{lemma}

\begin{proof}
In the case $m = 0$, this follows from the estimate $|P_l(\xi)| \leq 1$ \citep[Eq.\,3.2.2]{freeden_gervens_schreiner}. If $|m| \geq 1$, see \citep[Eq.\,5]{lohoefer}.
\end{proof}

\begin{lemma}\label{lem:legendre_odd}
Let\/ $1 \leq |m| \leq l$. Then
\begin{equation}\label{eq:legendre_special_recursion}
\frac{P_{lm}(\xi)}{\sqrt{1-\xi^2}} \,=\, - \frac{1}{2m} \Big\lbrace P_{l+1, m+1}(\xi) + (l - m + 1) \, (l - m + 2) \, P_{l+1, m-1}(\xi) \Big\rbrace.
\end{equation}
\end{lemma}

\begin{proof}
This known recursion formula can be derived from \citep[Eq.\,2.5.24]{edmonds}
\begin{equation}\label{eq:legendre_edmonds}
\frac{P_{lm}(\xi)}{\sqrt{1 - \xi^2}} \,=\, - \frac{1}{2 m \xi} \Big\lbrace P_{l, m+1}(\xi) + (l + m)(l - m + 1) P_{l,m-1}(\xi) \Big\rbrace
\end{equation}
and [\citeauthor{arfken_weber}, \citeyear{arfken_weber}, Eq.\,12.94; \citeauthor{edmonds}, \citeyear{edmonds}, Eq.\,2.5.25]
\begin{align}
\sqrt{1 - \xi^2} \, \frac{\mathrm{d} P_{lm}}{\mathrm{d}\xi} (\xi) \,&=\, - \frac{1}{2} \Big\lbrace P_{l, m+1}(\xi) - (l + m)(l - m + 1) P_{l, m-1}(\xi) \Big\rbrace,\label{eq:legendre_arfken_weber}\\
(1 - \xi^2) \, \frac{\mathrm{d}P_{lm}}{\mathrm{d}\xi} (\xi) \,&=\, (l + m) P_{l-1, m}(\xi) - l \xi P_{lm}(\xi).\label{eq:legendre_edmonds_2}
\end{align}
Indeed, from \eqref{eq:legendre_arfken_weber} and \eqref{eq:legendre_edmonds_2}, it follows by eliminating the derivative and with $l \mapsto l + 1$ that
\begin{equation}\label{eq:legendre_intermediate}
\frac{P_{lm}(\xi)}{\sqrt{1-\xi^2}} \,=\, - \frac{P_{l + 1, m + 1}(\xi)}{2(l + m + 1)} + \frac{l + 1}{l + m + 1} \frac{\xi P_{l + 1, m}(\xi)}{\sqrt{1 - \xi^2}} + \frac{l - m + 2}{2} P_{l + 1, m - 1}(\xi).
\end{equation}
Again with $l \mapsto l + 1$, it follows from \eqref{eq:legendre_edmonds} that
\begin{equation}\label{eq:legendre_intermediate_2}
\frac{\xi P_{l + 1, m}(\xi)}{\sqrt{1 - \xi^2}} \,=\, - \frac{1}{2m} \Big\lbrace P_{l + 1, m + 1}(\xi) + (l + m + 1)(l - m + 2) P_{l + 1, m - 1}(\xi) \Big\rbrace.
\end{equation}
Formula \eqref{eq:legendre_special_recursion} is now obtained by inserting \eqref{eq:legendre_intermediate_2} into \eqref{eq:legendre_intermediate} and simplifying.
\end{proof}

\begin{lemma}\label{lem:legendre_bound_odd}
For\/ $1 \leq |m| \leq l < n$, it is
\begin{equation*}
\sqrt{\frac{(l-m)!}{(l+m)!}} \, \frac{|P_{lm}(\xi)|}{\sqrt{1 - \xi^2}} \,\lesssim\, n.
\end{equation*}
\end{lemma}

\begin{proof}
With the triangle inequality it follows from Lemmata \ref{lem:legendre_bound_even} and \ref{lem:legendre_odd} that
\begin{align*}
\sqrt{\frac{(l-m)!}{(l+m)!}} \, \frac{|P_{lm}(\xi)|}{\sqrt{1 - \xi^2}} 
\,&\leq\, \frac{1}{2|m|} \, \sqrt{\frac{(l-m)!}{(l+m)!}} \, \Big\lbrace |P_{l+1,m+1}(\xi)| + (l - m + 1) (l - m + 2) |P_{l+1,m-1}(\xi)| \Big\rbrace\vphantom{\sqrt{\frac{!}{!}}}\\
%
%
\,&\leq\, \frac{1}{2|m|} \left( \sqrt{(l+m+1)(l+m+2)} + \sqrt{(l-m+1)(l-m+2)} \, \right)\\
\,&\leq\, \frac{1}{2|m|} \, (2l + 3),
\end{align*}
where the last estimate is due to the inequality of the arithmetic and the geometric mean. 
\end{proof}

\noindent
\textit{Proof of Theorem \ref{thm:nfsglft_error}}. In order to apply Theorem \ref{thm:nfft_error}, we investigate in Theorem \ref{thm:nfsglft} the impact of the factors on the right-hand side of the NDFT matrix $N$ on the $1$-norm of the input vector. For the block matrices, we can focus on the single blocks.

Firstly, we note that a multiplication of a vector $x$ by a permutation matrix $P$ has no impact on the $1$-norm, \textit{i.\,e.}, $\|Px\|_1 = \|x\|_1$. The same holds true for the matrices $A_B$, $W_m$, and $Z_B$. From Lemma \ref{lem:dct_1_norm}, it immediately follows that $\|\tilde{C}_{2B} x\|_1 \leq \sqrt{2B} \|x\|_1$ for all vectors $x$ of length $2B$. Furthermore, it is $\|D_{2B}x\|_1 \leq \|x\|_1 / \sqrt{B}$.
Due to Lemma \ref{lem:laguerre_bound} and the fact that $\rho (1 + \cos\omega_j) / 2 \leq \rho$, we have for the radial part of the SGL basis functions that
\begin{equation}\label{eq:radial_estimate}
\left|N_{nl} R_{nl} \left(\frac{\rho}{2}(1 + \cos\omega_j) \right)\right|^2
\leq\, \frac{2(n-l-1)!}{\Gamma(n + 1/2)} {n - 1/2 \choose n - l - 1}^2 \rho^{2l} \, \mathrm{e}^{\rho^2}
=\, \frac{2 \Gamma(n + 1/2) \rho^{2l} \mathrm{e}^{\rho^{2}}}{\Gamma(n-l) \Gamma(l + 3/2)^2}.
\end{equation}
In view of the right-hand side, we show that
\begin{align}
\ln \frac{\Gamma(n + 1/2) \rho^{2l}}{\Gamma(n-l) \Gamma(l + 3/2)^2}
\,&=\, \ln \frac{\Gamma(n+1/2)}{\Gamma(n-l)} + l\ln\rho^2 - 2 \ln \Gamma(l + 3/2)\vphantom{\left(\frac{2}{2}\right)^2}\notag\\
&\leq\, (l + 1/2) \ln(n + 1/2) + l\ln\rho^2 - 2\ln\left((l + 1)^{l+1} \mathrm{e}^{-l-1}\right) - \ln 2 - 1 \vphantom{\frac{\Gamma}{\Gamma}} \label{eq:nfsglft_error_many_logarithms}\\
\,&\leq\, (l + 1) \ln\Big((\rho\mathrm{e})^2 (n + 1/2) (l + 1)^{-2}\Big) - 2 \ln\rho - \ln 2 - 1 \vphantom{\frac{\Gamma}{\Gamma}} \notag\\
\,&\leq\, (\rho \mathrm{e})^{2/\mathrm{e}} \, (n + 1/2)^{1 / \mathrm{e}} \, (l + 1)^{1 - 2/\mathrm{e}} - 2 \ln\rho - \ln 2 - 1 \vphantom{\frac{\Gamma}{\Gamma}} \label{eq:nfsglft_error_many_logarithms_2}\\
\,&\leq\, (\rho \mathrm{e})^{2 / \mathrm{e}} \, (B + 1/2)^{1 - 1/\mathrm{e}} - 2 \ln\rho - \ln 2 - 1. \vphantom{\frac{\Gamma}{\Gamma}}\label{eq:nfsglft_error_many_logarithms_3}
\end{align}
Here, the estimate \eqref{eq:nfsglft_error_many_logarithms} follows from the Lemmata \ref{thm:nfsglft_error_lem_1} and \ref{thm:nfsglft_error_lem_2}, while the inequality \eqref{eq:nfsglft_error_many_logarithms_2} is due to the Lemma \ref{thm:nfsglft_error_lem_3}. For $\rho \ll 1$, the estimate \eqref{eq:nfsglft_error_many_logarithms_3} can be improved: Omitting in \eqref{eq:nfsglft_error_many_logarithms} the summand $l \ln \rho^2 \leq 0$, we obtain similarly as above 
\begin{equation}
\ln \frac{\Gamma(n+1/2) \rho^{2l}}{\Gamma(n-l) \Gamma(l + 3/2)^2} \,\leq\, \mathrm{e}^{2 / \mathrm{e}} \, (B + 1/2)^{1 - 1/\mathrm{e}} - \ln 2 - 1.\label{eq:nfsglft_error_many_logarithms_4}
\end{equation}
A comparison of the right-hand side of \eqref{eq:nfsglft_error_many_logarithms_3} with \eqref{eq:nfsglft_error_many_logarithms_4} then shows that
\begin{equation*}
\mathrm{e}^{2 / \mathrm{e}} \, (B + 1/2)^{1 - 1/\mathrm{e}} \,\leq\, (\rho \mathrm{e})^{2 / \mathrm{e}} \, (B + 1/2)^{1 - 1/\mathrm{e}} - 2 \ln \rho
\end{equation*}
if and only if $B \leq \Omega(\rho)$ with $\Omega$ as stated in the theorem.
By \eqref{eq:radial_estimate}, \eqref{eq:nfsglft_error_many_logarithms_3}, and \eqref{eq:nfsglft_error_many_logarithms_4}, for fixed $l < B$ and an arbitrary vector $x = [\xi_0, \dots, \xi_{B-l-1}]$ of length $B - l$, the elements of the transformed vector $\tilde{R}_l^\textnormal{T} x$ can hence be estimated as
\begin{align*}
\left| \sum_{n = l+1}^B \xi_{n-l-1} \, N_{nl} \, R_{nl} \left(\frac{\rho}{2}(1 + \cos\omega_j) \right)\right|
\,&\leq\, a_{\rho, B} \, \exp\bigg(b_{\rho, B} \left(B + \frac{1}{2}\right)^{1 - 1 / \mathrm{e}} + \frac{\rho^2}{2} - \frac{1}{2}\bigg) \, \|x\|_1,
\end{align*}
with the coefficients $a_{\rho, B}$ and $b_{\rho, B}$ stated in the theorem. We thus have for a vector $x$ of appropriate length that
\begin{equation}\label{eq:nfsglft_error_R}
\|\tilde{R}_l^\textnormal{T} x \|_1 \,\lesssim\, B \, a_{\rho, B} \, \exp\bigg(b_{\rho, B} \left(B + \frac{1}{2}\right)^{1 - 1 / \mathrm{e}} + \frac{\rho^2}{2} - \frac{1}{2}\bigg) \, \|x\|_1, \quad\quad l < B.
\end{equation}

Let now $m$ be odd with $|m| < B$, and $x = [\xi_0, \dots, \xi_{B-|m|}]$ be a vector of length $B - |m|$. With Lemma \ref{lem:legendre_bound_even}, we can estimate the elements of the transformed vector $\tilde{L}_m^\textnormal{T} x$ as
\begin{align*}
\Bigg| \sum_{l = |m|}^{B-1} \xi_{l-|m|} \, Q_{lm} \, P_{lm}(\cos\omega_j) \Bigg|
\,\lesssim\, \sqrt{B} \, \|x\|_1.
\end{align*}
When $m$ is odd, on the other hand, it follows from Lemma \ref{lem:legendre_bound_odd} that
\begin{align*}
\Bigg| \sum_{l = |m|}^{B-1} \xi_{l-|m|} \, (\sin\omega_j)^{-1} \, Q_{lm} \, P_{lm}(\cos\omega_j) \Bigg|
\,\lesssim\, B^{3 / 2} \, \|x\|_1.
\end{align*}
We thus have for a vector $x$ of appropriate length the estimate
\begin{equation}\label{eq:nfsglft_error_L}
\|\tilde{L}_m^\textnormal{T} x\|_1 \,\lesssim\, B^{5 / 2} \, \|x\|_1, \quad\quad |m| < B.
\end{equation}
With the estimates \eqref{eq:nfsglft_error_R} and \eqref{eq:nfsglft_error_L}, we can now apply the second part of Theorem \ref{thm:nfft_error}.\hfill$\Box$\newline

Of course, Theorem \ref{thm:nfsglft_error} suggests that the error $E_\infty$ grows when the bandwidth $B$ is increased. The same applies to increasing parameter $\rho$, which motivates in Section \ref{sec:nfsglft} the choice
\begin{equation}\label{eq:nfsglft_rho}
\rho \,=\, \rho(x_0,\dots,x_{M-1}) \,\coloneqq\, \max \lbrace r_0, \dots, r_{M-1} \rbrace.
\end{equation}
However, when the bandwidth $B$ and the parameter $\rho$ are fixed, we can make the error \emph{arbitrarily small} (in exact arithmetics) by choosing the cutoff parameter $q$ of the NFFT sufficiently large. More specifically, we observe that the maximum absolute error $E_\infty$ of the NFSGLFT\vskip0.5\baselineskip
\begin{enumerate}[label=(O\arabic*), align=left]\setlength{\itemsep}{4pt}
\item decays at least \emph{exponentially} with increasing cutoff parameter $q$,\label{itm:O1}
\item grows potentially \emph{hyperexponentially} with increasing parameter $\rho$,\label{itm:O2}
\item grows at most \emph{subexponentially} with increasing bandwidth $B$.\label{itm:O3} 
\end{enumerate}\vskip0.5\baselineskip
We validate these observations numerically in the upcoming Section \ref{sec:numerical_experiments}. In particular, (iii) implies that in order to control the error when the bandwidth $B$ is increasing, it suffices to appropriately choose $\sigma = \sigma(B) = \textnormal{constant}$ and $q = q(B) = \mathrm{o}(B)$. This shows that the NFSGLFT derived in Section \ref{sec:nfsglft} are truly fast algorithms from the theoretical viewpoint (cf.\ Eqs.\,\ref{eq:Phi_Psi}).

\section{Numerical results}
\label{sec:numerical_experiments}
The NFSGLFT, its adjoint, as well as the iNFSGLFT were implemented in the C++ programming language. We employed the Clenshaw-Smith algorithm and its adjoint (Sect.\,\ref{sec:clenshaw}) in the radial subtransform, and the semi-naive FLT and its adjoint (Sect.\,\ref{sec:flt}) from the software package SpharmonicKit\footnote{\url{http://www.cs.dartmouth.edu/~geelong/sphere}} (Vers.\,2.7) in the spherical subtransform, respectively. The Clenshaw-Smith algorithm and its adjoint were performed in extended double precision (\texttt{long double}, 80\,bits in total, 64\,bits mantissa). The other computations were performed in standard double precision. We used the implementation of the NFFT and its adjoint of \citet{keiner_kunis_potts}\footnote{\url{https://www-user.tu-chemnitz.de/~potts/nfft}} (Vers.\,3.3.1) and employed the DCT of the Fastest Fourier Transform in the West\footnote{\url{http://www.fftw.org}} (FFTW, Vers.\,3.3.6). All test runs were performed on an x86-64 Unix system with a 3.40\,GHz Intel Core i7-3770 CPU. We chose the oversampling factor $\sigma = 2$ in the NFFT. The parameter $\rho$ of the NFSGLFT and its adjoint were set as in \eqref{eq:nfsglft_rho}. For runtime and error comparison, a naive NDSGLFT was implemented. This naive NDSGLFT was realized by directly evaluating the function of interest at the given points, using the implementation of the generalized Laguerre polynomials, associated Legendre polynomials, etc.\ of the GNU Scientific Library\footnote{\url{https://www.gnu.org/software/gsl}} (GSL, Vers.\,2.3). In the following, it is
\begin{equation*}
\mathbb{B}^3_\kappa \,\coloneqq\, \lbrace x \in \mathbb{R}^3 : \|x\|_2 \leq \kappa \rbrace, \quad\quad \kappa > 0.
\end{equation*}

\begin{remark}
Reviewing the derivation of the NFSGLFT in Section \ref{sec:nfsglft} carefully, it becomes apparent that not the entire grid $I^3_{4B}$ of size $4B \times 4B \times 4B$ is actually required in \eqref{eq:nfsglft_nfft}. Since the implementation of the NFFT of \citeauthor{keiner_kunis_potts}\ allows for using grids with a different extent in each direction, we used a smaller grid of size $4B \times 2B \times 2B$ in the test runs. This generally improves the runtime, but has no impact on the asymptotic complexity.
\end{remark}

\begin{figure}[t]
\centering
\includegraphics[width=0.95\textwidth]{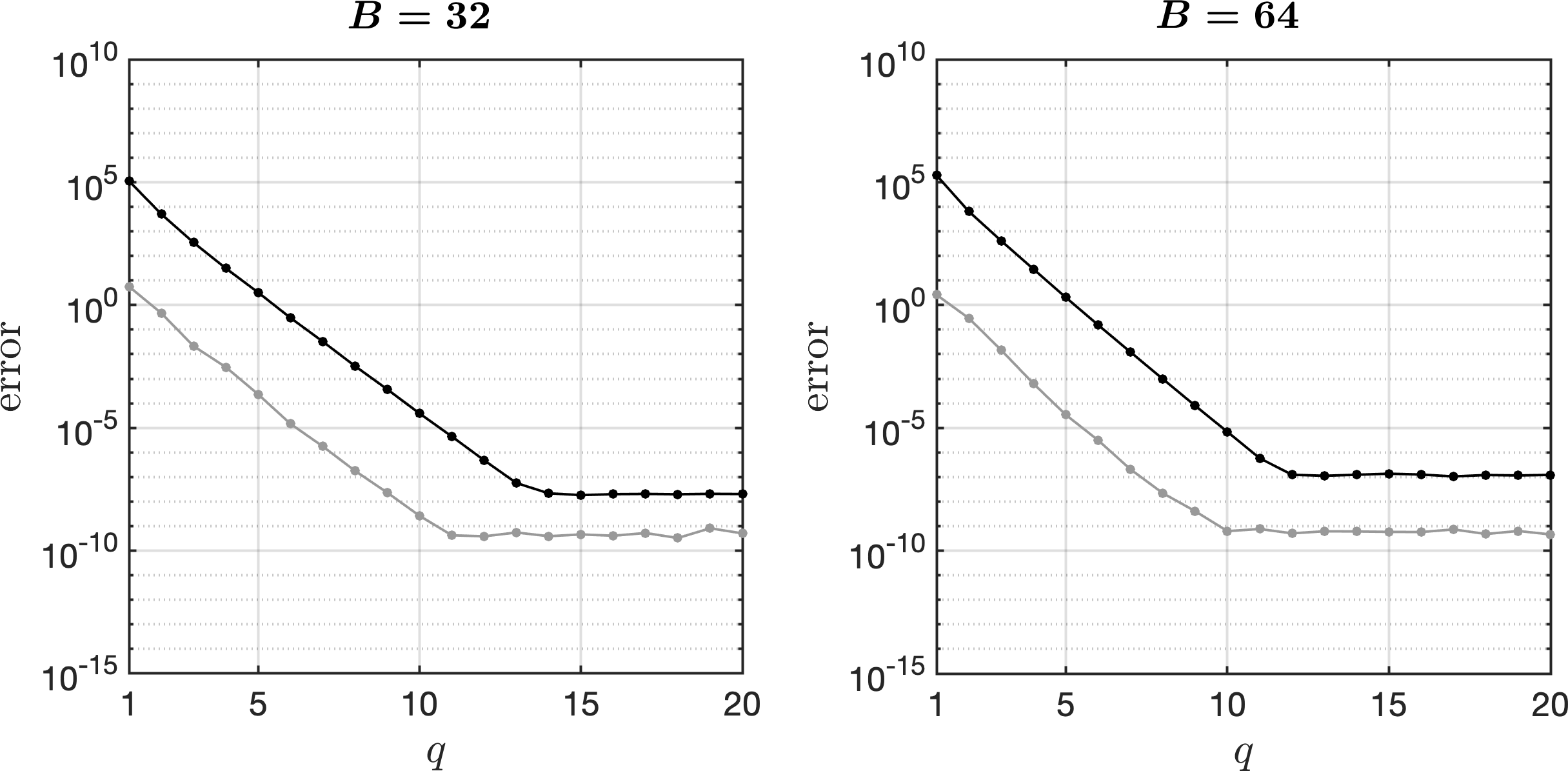} 
\caption{Average maximum \textbf{(black)} absolute and \textbf{(gray)} relative error of the NFSGLFT with $M = 10\,000$ uniformly distributed points in $\mathbb{B}^3_5$ and $\sigma = 2$ \textit{vs.}\ the cutoff parameter $q$ of the NFFT.}\label{fig:nfsglft_q}
\end{figure}

In a first test, random SGL Fourier coefficients $\hat{f}_{nlm}$ were generated for the bandwidths $B = 32$ and $B = 64$, respectively, the real and imaginary part both being uniformly distributed between $-1$ and $1$. In addition, $M = 10\,000$ uniformly distributed random points $x_i \in \mathbb{B}^3_5$ were generated each. The corresponding function values $f(x_i)$ were then computed using the exact naive NDSGLFT, while the NFSGLFT was used to compute approximative function values $\tilde{f}(x_i)$. This was done for the cutoff parameters $q = 1, \dots, 20$ of the NFFT. The above test was repeated ten times to determine the unweighted average maximum absolute and relative error of the NFSGLFT, these errors being defined respectively as (cf.\ Eq.\,\ref{eq:nfsglft_error_def})
\begin{equation*}
\max_{i = 0, \dots, M-1} |f(x_i) - \tilde{f}(x_i)|
\quad \textnormal{and} \quad
\max_{i = 0, \dots, M-1} \frac{|f(x_i)-\tilde{f}(x_i)|}{|f(x_i)|}.
\end{equation*}
Figure \ref{fig:nfsglft_q} shows the results of this error measurement (the standard deviation of the results was generally so small, that it was not drawn into the plot for better visual perception). It can clearly be seen that the error decays exponentially w.r.t.\ the cutoff parameter $q$ (cf.\ Sec.\,\ref{sec:error}, \hyperref[itm:O1]{O1}), until the roundoff error takes over. In both cases $B = 32$ and $B = 64$, the relative error for the power of two $q = 16$ is on the order of only $10^{-10}$, which is satisfyingly small, taking into account the machine accuracy.

In a second test, random SGL Fourier coefficients $\hat{f}_{nlm}$ were generated as described above for the bandwidths $B = 32$ and $B = 64$, respectively. In addition, for each $k = 1, \dots, 6$, a total of $M = 10^k$ uniformly distributed random points $x_i \in \mathbb{B}^3_5$ was generated. The corresponding function values $f(x_i)$ were then computed with the naive NDSGLFT as well as the NFSGLFT. Table \ref{tab:nfsglft_runtime} shows the results of the runtime measurement performed in this test. It is clearly visible that the NFSGLFT can offer a significant runtime advantage over the naive algorithm even for small problem sizes ($M \geq 100$). With increasing $M$, the NFSGLFT improves even further. In the case $B = 32$ and $M = 1\,000\,000$, the naive NDSGLFT required approximately three hours of computation time, the runtime of the NFSGLFT was less than two minutes.

\begin{table}[t]
\begin{subtable}[t]{0.45\textwidth}
\begin{center}
\begin{tabular}[t]{ccccccc}
$M$ & naive NDSGLFT & NFSGLFT\\\hline
& &\\[-11pt]   
1\,\small{E}\,\normalsize{$+$1} \,&\, 9.99\,\small{E}\,\normalsize{$-$2}\,s \,&\, 9.27\,\small{E}\,\normalsize{$-$}1\,s\\[2pt]
1\,\small{E}\,\normalsize{$+$2} \,&\, 9.95\,\small{E}\,\normalsize{$-$1}\,s \,&\, 9.35\,\small{E}\,\normalsize{$-$}1\,s\\[2pt]
1\,\small{E}\,\normalsize{$+$3} \,&\, 9.81\,\small{E}\,\normalsize{$+$0}\,s \,&\, 1.03\,\small{E}\,\normalsize{$+$}0\,s\\[2pt]
1\,\small{E}\,\normalsize{$+$4} \,&\, 1.04\,\small{E}\,\normalsize{$+$2}\,s \,&\, 2.10\,\small{E}\,\normalsize{$+$}0\,s\\[2pt]
1\,\small{E}\,\normalsize{$+$5} \,&\, 1.09\,\small{E}\,\normalsize{$+$3}\,s \,&\, 1.22\,\small{E}\,\normalsize{$+$}1\,s\\[2pt]
1\,\small{E}\,\normalsize{$+$6} \,&\, 1.11\,\small{E}\,\normalsize{$+$4}\,s \,&\, 1.13\,\small{E}\,\normalsize{$+$}2\,s\\\hline
\end{tabular}
\end{center}
\end{subtable}\quad\quad
\begin{subtable}[t]{0.45\textwidth}
\begin{center}
\begin{tabular}[t]{ccccccc}
$M$ & naive NDSGLFT & NFSGLFT\\\hline
& &\\[-11pt]   
1\,\small{E}\,\normalsize{$+$1} \,&\, 1.36\,\small{E}\,\normalsize{$+$0}\,s \,&\, 1.01\,\small{E}\,\normalsize{$+$1}\,s\\[2pt]
1\,\small{E}\,\normalsize{$+$2} \,&\, 1.38\,\small{E}\,\normalsize{$+$1}\,s \,&\, 1.03\,\small{E}\,\normalsize{$+$1}\,s\\[2pt]
1\,\small{E}\,\normalsize{$+$3} \,&\, 1.36\,\small{E}\,\normalsize{$+$2}\,s \,&\, 1.03\,\small{E}\,\normalsize{$+$1}\,s\\[2pt]
1\,\small{E}\,\normalsize{$+$4} \,&\, 1.36\,\small{E}\,\normalsize{$+$3}\,s \,&\, 1.19\,\small{E}\,\normalsize{$+$1}\,s\\[2pt]
1\,\small{E}\,\normalsize{$+$5} \,&\, 1.37\,\small{E}\,\normalsize{$+$4}\,s \,&\, 2.82\,\small{E}\,\normalsize{$+$1}\,s\\\hline
\end{tabular}
\end{center}
\end{subtable}\vskip\baselineskip
\caption{Average runtime of the naive NDSGLFT and the NFSGLFT for \textbf{(left)} $B = 32$ and \textbf{(right)} $B = 64$, each with $\sigma = 2$ and $q = 16$, \textit{vs.}\ the total number $M$ of target points.}\label{tab:nfsglft_runtime}
\end{table}

In the next test, the error of the NFSGLFT w.r.t.\ the spreading width of the target points was measured. For this, random SGL Fourier coefficients $\hat{f}_{nlm}$ were generated for the bandwidth $B = 32$ as described above. Furthermore, $M = 1000$ uniformly distributed random points $x_i \in \mathbb{B}^3_\kappa$ were generated, where $\kappa$ was iterated over the values  $i / 4$, $i = 1, \dots, 31$. Note that due to \eqref{eq:nfsglft_rho}, it is $\rho \approx \kappa$ here, as can be seen easily. The above test was performed ten times in order to determine the average maximum absolute and relative error of the NFSGLFT w.r.t.\ $\kappa$. The results are depicted in Figure \ref{fig:nfsglft_scattering}. The conjectured hyperexponential error growth w.r.t.\ the spreading width $\rho$ can clearly be observed (cf.\ Sec.\,\ref{sec:error}, \hyperref[itm:O2]{O2}). The standard deviation of the maximum absolute error was again so small that it was not drawn into Figure \ref{fig:nfsglft_scattering}.

\begin{figure}[t]
\centering
\includegraphics[width=\textwidth]{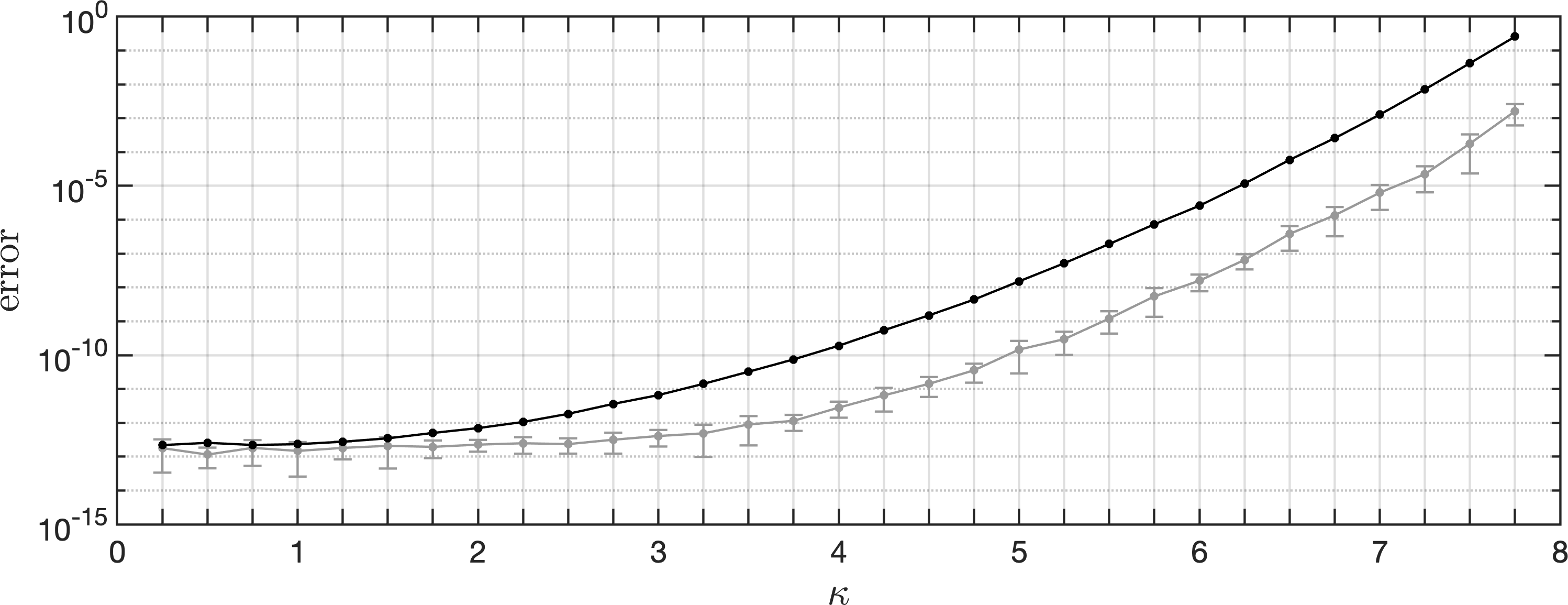}
\caption{Average maximum \textbf{(black)} absolute and \textbf{(gray)} relative error of the NFSGLFT for $B = 32$ and $M = 1000$ uniformly distributed random points in $\mathbb{B}^3_\kappa$ with $\sigma = 2$ and $q = 16$ \textit{vs.}\ $\kappa$ ($\approx \rho$).}
\label{fig:nfsglft_scattering}
\end{figure}

In a fourth test, the error of the NFSGLFT w.r.t.\ the bandwidth $B$ was investigated. To do so, for the bandwidths $B = 2^k$, $k = 3, \dots, 7$, random SGL Fourier coefficients $\hat{f}_{nlm}$ were generated as above. In addition, $M = 1000$ uniformly distributed random points $x_i \in \mathbb{B}^3_5$ were generated. The cutoff parameter of the NFFT was set to $q = 12$. In the actual test run, the function values $f(x_i)$ were computed from the SGL Fourier coefficients with the naive NDSGLFT as well as the NFSGLFT. This test was performed ten times. Figure \ref{fig:nfsglft_B} shows the results of the error measurement. One can see the subexponential growth of the error (cf.\ Sec.\,\ref{sec:error}, \hyperref[itm:O3]{O3}). Surprisingly, the absolute error decays from the bandwidth $B = 32$ to the bandwidth $B = 64$. By repeating the above test with the exact NDFT instead of the approximating NFFT, it became apparent that the reason for this reproducible effect was indeed the NFFT (cf.\ Fig.\,\ref{fig:nfsglft_B}).

\begin{figure}[t]
 \centering
 \includegraphics[width=\textwidth]{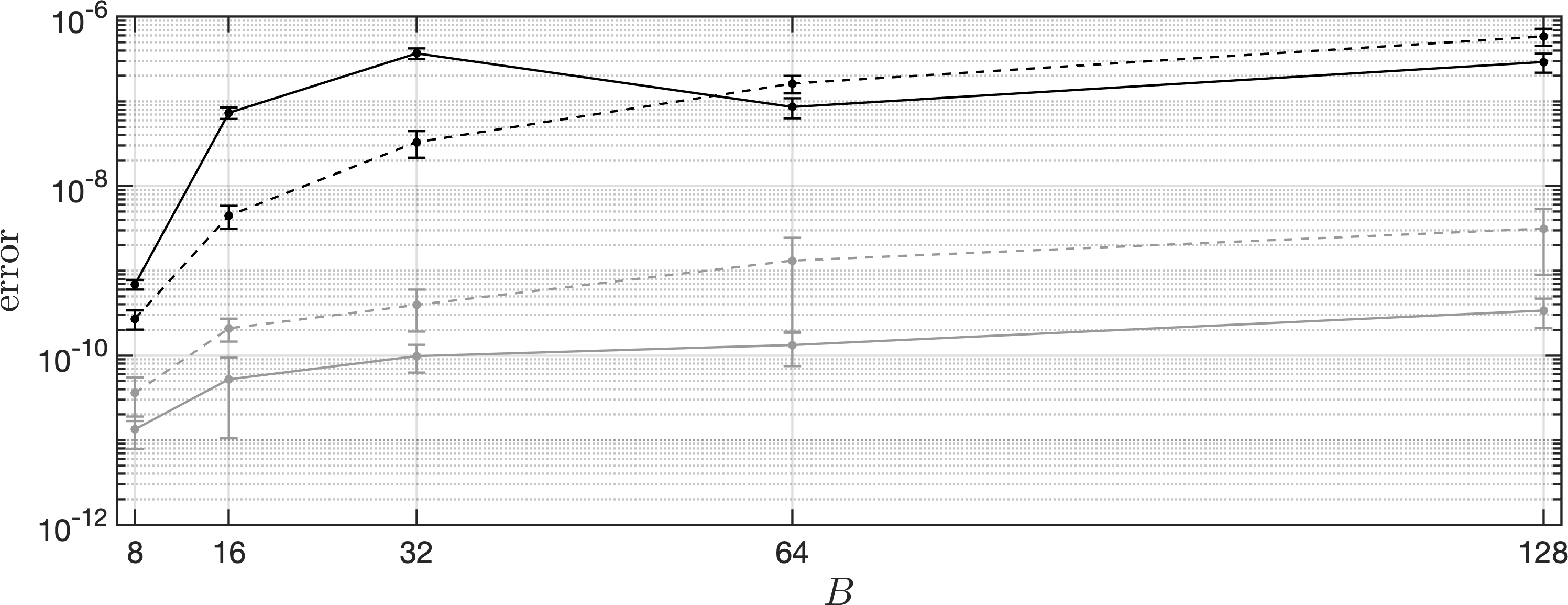}
 \caption{Average maximum \textbf{(black)} absolute and \textbf{(gray)} relative error of the NFSGLFT with $M = 1000$ uniformly distributed random points in $\mathbb{B}^3_5$ and $\sigma = 2$ and $q = 12$ \textit{vs.}\ the bandwidth $B$. The dashed lines show the results of the test repeated with the NDFT instead of the NFFT.}
 \label{fig:nfsglft_B}
\end{figure}

In a last test, the performance of the iNFSGLFT was examined. We used the Cartesian grids
\begin{equation*}
G_N \,\coloneqq\, \left\lbrace \kappa \cdot \left[\frac{2 j}{N} - 1, \frac{2 k}{N} - 1, \frac{2 l}{N} - 1 \right] \in \mathbb{R}^3 : j, k, l = 0, \dots, N - 1 \right\rbrace, \quad\quad \kappa > 0, ~ N \in \mathbb{N}.
\end{equation*}
It is here $M = N^3$. For the bandwidths $B = 8$ ($q = 15$) and $B = 16$ ($q = 16$), and with $\kappa = 5$, the cases $N = 25, 50, 100$ were considered. Further, for $N = 50$, the cases $\kappa = 6, 8, 10$ were investigated. Since in all above cases the number $M$ of target points is larger than the number of SGL Fourier coefficients, the iNFSGLFT was realized as a CGNR method (cf.\ Sect.\,\ref{sec:nfft}). Within this CGNR algorithm, the coefficients $\alpha_k$ and $\beta_k$ (see \citep[Alg.\,10.4.1]{golub_van_loan}) were computed using extended double precision. In the actual test runs, random SGL Fourier coefficients $\hat{f}_{nlm}$ were generated as above. The corresponding function values $f(x_i)$ were then computed with the exact naive NDSGLFT. From these data, SGL Fourier coefficients $\hat{f}^\circ_{nlm}$ were reconstructed with the iNFSGLFT. As the initial guess for the SGL Fourier coefficients required in the CGNR technique, the \emph{mid-point rule} 
\begin{equation*}
\hat{f}_{nlm} \,\approx\, 8 N^{-1} \sum\nolimits_{x_i \in G_N} f(x_i) \, \overline{H_{nlm}(x_i)} \, \mathrm{e}^{-\|x_i\|_2^2}
\end{equation*}
for numerical integration was used. After each iteration of the CGNR algorithm, the maximum absolute and relative error of the iNFSGLFT,
\begin{equation*}
\max_{|m| \leq l < n \leq B} |\hat{f}_{nlm} - \hat{f}_{nlm}^\circ|
\quad \textnormal{and} \quad
\max_{|m| \leq l < n \leq B} \frac{|\hat{f}_{nlm} - \hat{f}_{nlm}^\circ|}{|\hat{f}_{nlm}|},
\end{equation*}
respectively, were measured. The results are shown in Figures \ref{fig:infsglft_N} and \ref{fig:infsglft_gamma}. For $B = 8$, $\kappa = 5$ (Fig.\,\ref{fig:infsglft_N}, left), a small error was achieved for all considered values of $N$, though many iterations were necessary for this. Generally, except for in the cases  $B = 8$, $\kappa = 8, 10$ (Fig.\,\ref{fig:infsglft_gamma}, left), it can be observed that even after ten thousand iterations convergence was not attained in the iNFSGLFT. For $B = 16$, $\kappa = 5$ (Fig.\,\ref{fig:infsglft_N}, right), decay of the absolute error can be seen at the end of the test run, a relative error of less than one was not achieved, however; this can be attributed partially to the fact that this was not the case in the initial guess of the SGL Fourier coefficients, either. Overall, it seems that the convergence behavior is not influenced much by the number $M$ of given function values $f(x_i)$, provided that it is possible to reconstruct the SGL Fourier coefficients $\hat{f}_{nlm}$ from these scattered data. The spreading width of the function values $f(x_i)$ appears to have a much greater impact (Fig.\,\ref{fig:infsglft_gamma}). This became apparent already in the initial guess for the SGL Fourier coefficients $\hat{f}_{nlm}$; both for $B = 8$ and $B = 16$, the initial guess was significantly better for $\kappa = 8$ than for $\kappa = 6$ and $\kappa = 10$, for $\kappa = 10$ it was better than for $\kappa = 6$. Interestingly, in the case $B = 8$ (Fig.\,\ref{fig:infsglft_gamma}, left), the error started to grow after a certain number of iterations for $\kappa = 8$ and $\kappa = 10$; it is thus important to note that despite this error growth, the \emph{residual} within the CGNR method was decreasing. This is an indicator for the ill-posedness of the problem. In the case $B = 16$ (Fig.\,\ref{fig:infsglft_gamma}, right), this phenomenon was not observed. Here, in the case $\kappa = 6$, again no relative error of less than one was achieved. In the cases $\kappa = 8$ and $\kappa = 10$, on the other hand, the method was more successful.

\begin{figure}[t!]
 \centering
 \includegraphics[width=\textwidth]{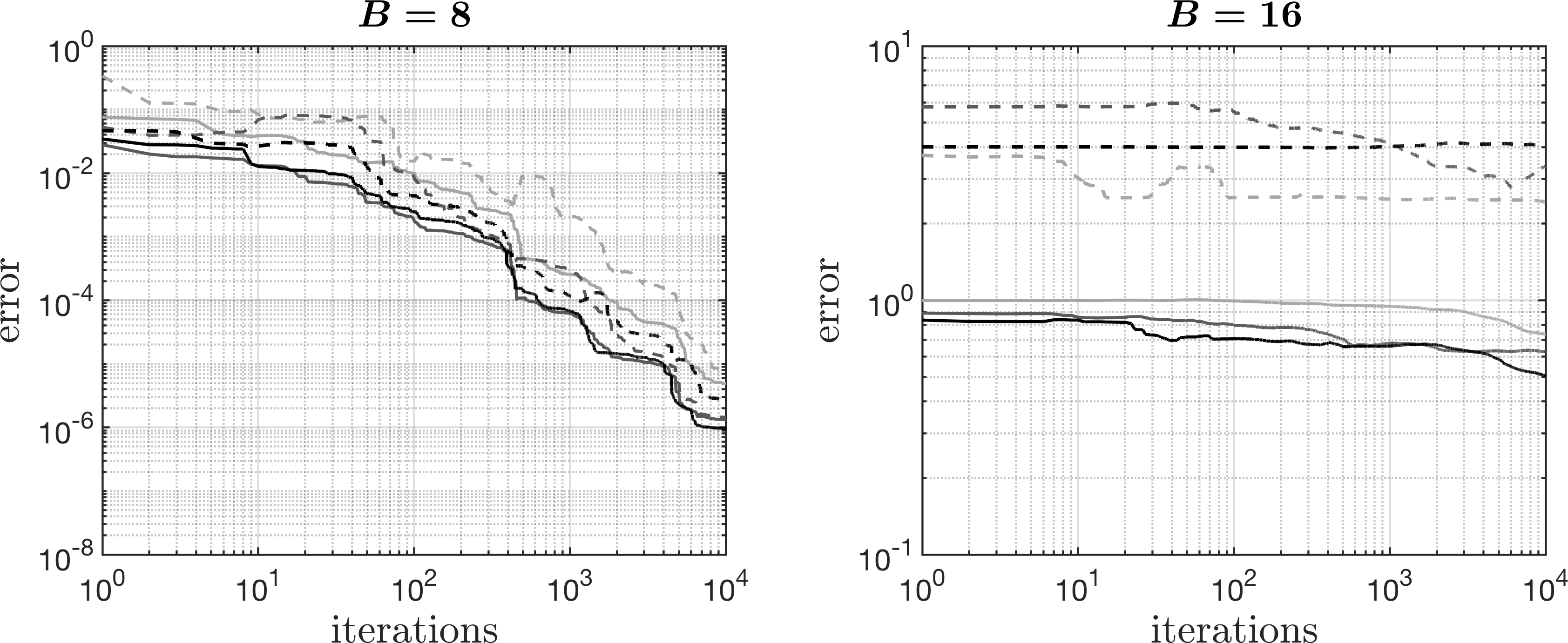}
 \caption{Maximum \textbf{(solid)} absolute and \textbf{(dashed)} relative error of the iNFSGLFT \textit{vs.}\ the number of iterations, with $\kappa = 5$ and \textbf{(light gray)} $N = 25$, \textbf{(gray)} $N = 50$, and \textbf{(black)} $N = 100$.}
 \label{fig:infsglft_N}\vskip\baselineskip
 \centering
 \includegraphics[width=\textwidth]{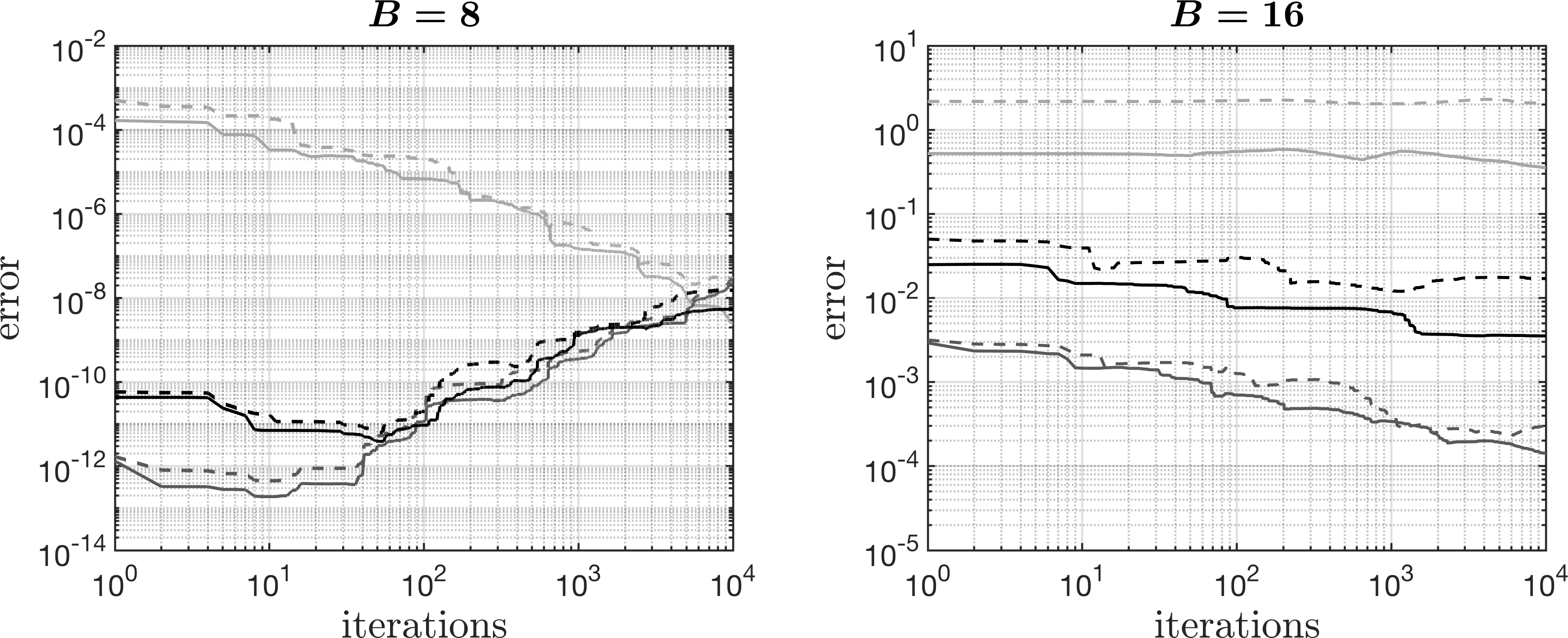}
 \caption{Maximum \textbf{(solid)} absolute and \textbf{(dashed)} relative error of the iNFSGLFT \textit{vs.}\ the number of iterations, with $N = 50$ and \textbf{(light gray)} $\kappa = 6$, \textbf{(gray)} $\kappa = 8$, and \textbf{(black)} $\kappa = 10$.} \label{fig:infsglft_gamma}
\end{figure}

In summary, the above results clearly demonstrate that the NFSGLFT is a practical class of fast algorithms that can offer a significant runtime advantage of less than half a minute as opposed to almost four hours in the case $B = 64$ and $M = 100\,000$, for example. The error of these approximating algorithms is relatively small, provided that the spreading width of the target points $x_i$\,--\,and thus the radial parameter $\rho$\,--\,is not too large. The iNFSGLFT constructed from the NFSGLFT and its adjoint achieved a good result in some cases, but further developments are necessary to improve the convergence behavior. The problem of slow convergence in the CGNR and CGNE methods is well known (cf.\ \citep[p.\,546]{golub_van_loan}). A starting point for further developments are the considerations and techniques in \citep[Chap.\,5]{kunis} and \citep{kircheis_potts}.
A possible explanation for the observed instability of our fast transforms 
w.\,r.\,t.\ the radial parameter $\rho$ is that, contrary to orthogonal polynomial recurrence on a compact subset of the real line, the forward and backward Laguerre recurrence is unstable for large arguments, because the interval containing the roots of the Laguerre polynomials increases with the degree. This means that for large degree, the polynomials may be absolutely small at fixed argument, but if this argument is also large, then the low-degree polynomials are bound to be absolutely large in comparison. This could be problematic because as the bandwidth $B$ increases, one might anticipate or even require more samples further from the origin in practice. Further research will address the interplay between the radial parameter $\rho$ and the bandwidth $B$.

\section*{Acknowledgements}
The author would like to thank the referees for their very valuable comments. The second referee provided the potential explanation for the observed instability of the transforms w.\,r.\,t.\ the radial parameter $\rho$ above, and suggested further research on the connection between the bandwidth $B$ and $\rho$. Furthermore, the author would like to thank J\"urgen Prestin and Daniel Potts for scientific discussion. 

\bibliography{references}
\bibliographystyle{my_apa}

\end{document}